\documentclass[a4paper]{amsart}

\usepackage[all]{xy}

\usepackage{pstricks-add}
\usepackage{float}

\usepackage[colorlinks=true]{hyperref}
\usepackage{enumerate}
\usepackage{amssymb}
\usepackage{amsbsy}
\usepackage{comment}
\usepackage{aurical}
\usepackage{graphicx}
\usepackage{mathrsfs}
\usepackage[cal=boondox]{mathalfa}

\newtheorem{thm}{Theorem}[section]
\newtheorem{prop}[thm]{Proposition}
\newtheorem{lem}[thm]{Lemma}
\newtheorem{cor}[thm]{Corollary}

\theoremstyle{definition}
\newtheorem{defn}[thm]{Definition}

\theoremstyle{remark}
\newtheorem{rem}[thm]{Remark}

\newcommand{\ds}{\displaystyle}                         % displaystyle

\usepackage{color}
\usepackage{xcolor}
\usepackage{framed}

\definecolor{DarkBlue}{rgb}{0.1,0.1,0.5}
\definecolor{Red}{rgb}{0.9,0.0,0.1}
\definecolor{lightred}{rgb}{1, 0.75, 0.75}

\colorlet{light-yellow}{orange!60!yellow}
\colorlet{light-blue}{blue!15}
\colorlet{darkblue}{blue!70!black}

\usepackage{scalerel}[2014/03/10]
\usepackage{stackengine}

\newcommand\reallywidetilde[1]{\ThisStyle{%
  \setbox0=\hbox{$\SavedStyle#1$}%
  \stackengine{-.1\LMpt}{$\SavedStyle#1$}{%
    \stretchto{\scaleto{\SavedStyle\mkern.2mu\sim}{.5467\wd0}}{.7\ht0}%
  }{O}{c}{F}{T}{S}%
}}
\begin{document}

\title[Higher multi-Courant algebroids]{Higher multi-Courant  algebroids}
\author{P. Antunes}
\address{CMUC, Department of Mathematics, University of Coimbra, 3001-501 Coimbra, Portugal}
\email{pantunes@mat.uc.pt}
\author{J.M. Nunes da Costa}
\address{CMUC, Department of Mathematics, University of Coimbra, 3001-501 Coimbra, Portugal}
\email{jmcosta@mat.uc.pt}
\begin{abstract}
The binary bracket of a Courant algebroid structure on $(E,\langle \cdot,\cdot \rangle)$ can be extended to a $n$-ary bracket on $\Gamma(E)$, yielding a multi-Courant algebroid. These $n$-ary brackets form a Poisson algebra and were defined, in an algebraic setting, by Keller and Waldmann. We construct a higher geometric version of Keller-Waldmann Poisson algebra and define higher multi-Courant algebroids. As Courant algebroid structures can be seen as degree $3$ functions on a graded symplectic manifold of degree $2$, higher multi-Courant structures  can be seen as functions of degree $n\geq 3$ on that graded symplectic manifold.
\end{abstract}

%%% ----------------------------------------------------------------------
\maketitle
%%% ----------------------------------------------------------------------


\


%%%%%%%%%%%%%%%%%%%%%%%%%%%%%%%%%%%%
%%%%%%%%%%%%%%%%%%%%%%%%%%%%%%%%%%%%
%%%%%%%%%%%%%%%%%%%%%%%%%%%%%%%%%%%%
%
\section{Introduction}
Aiming at  interpreting the bracket on the Whitney sum $TM \oplus T^*M$ of the tangent and cotangent bundle of a smooth manifold $M$, proposed by Courant in \cite{Courant}, Liu, Weinstein and Xu \cite{Liu_Weinstein_Xu} introduced the concept of Courant algebroid on a vector bundle $E\to M$. This vector bundle is equipped with a fiberwise symmetric bilinear form $\langle \cdot, \cdot \rangle$, a Leibniz bracket on the space $\Gamma(E)$ of sections and a morphism of vector bundles $\rho:E \to TM$, called the anchor, satisfying a couple of compatibility conditions. In \cite{royContemp}, Roytenberg  described a Courant algebroid as a degree $2$ symplectic graded manifold $\mathcal{F}_E$ together with a degree $3$ function $\Theta$ satisfying $\{\Theta, \Theta \}=0$, where $\{\cdot , \cdot \}$ is the graded Poisson bracket corresponding to the graded symplectic structure.  The morphism $\rho$ and the Leibniz bracket on $\Gamma(E)$ are recovered as derived brackets (see \ref{subsection2.2}).

The Courant bracket, or its no skew-symmetric version called Dorfman bracket, is a binary bracket. The first attempt to extend it to a $n$-ary bracket was given, in purely algebraic terms, by Keller and Waldmann in \cite{Keller-Waldmann}.  They built a graded Poisson algebra $\mathcal{C}=\oplus_{n\geq 0} \mathcal{C}^n$ of degree $-2$, with $\mathcal{C}^{n+1}$ the collection of $n$-ary brackets. Elements of  $\mathcal{C}^{3}$ that are closed with respect to the graded Poisson bracket correspond to Courant structures. The graded Poisson algebra $\mathcal{C}$, that we call Keller-Waldmann Poisson algebra, is a complex that controls deformation. Keller-Waldmann algebra elements are $n$-ary  brackets  and each bracket comes with a symbol. In degree $3$, the symbol is the anchor of the Courant structure.

We consider the geometric counterpart of the Keller-Waldmann Poisson algebra, and our starting point is a vector bundle $E \to M$ equipped with a fiberwise symmetric bilinear form $\langle \cdot, \cdot \rangle$. This is also the setting in \cite{A10}, where the first author has started the study of the  Keller-Waldmann algebra under a  geometric point of view. In this case, the Keller-Waldmann Poisson algebra is denoted by $\mathcal{C}(E)= \oplus_{n\geq 0}\mathcal{C}^n(E)$ and its elements are pre-multi-Courant brackets  on $\Gamma(E)$. The prefix \emph{pre} means that  elements  $C\in \mathcal{C}(E)$ do not need to close with respect to the Poisson bracket, denoted by  $[\cdot,\cdot]_{_{K\!W}}$. If $[C,C]_{_{K\!W}}=0$, the triple $(E, \langle \cdot, \cdot \rangle, C)$ is a multi-Courant algebroid. At this point it is important to notice that, for $n\neq2$,  what is called $n$-Courant bracket in Remark 3.2 of \cite{Keller-Waldmann} is not the same as our $n$-ary Courant bracket,  because we require the closedness with respect to  the $[\cdot,\cdot]_{_{K\!W}}$ bracket, while in \cite{Keller-Waldmann} the authors ask the closedness with respect to a  different bracket. For $n=2$, the two brackets coincide (see Remark~\ref{rem_Filippov}).

As it is noticed in \cite{Keller-Waldmann}, each pre-multi-Courant bracket $C \in \mathcal{C}^n(E)$ determines a map $ \reallywidetilde{C} :  \Gamma(E) \times \stackrel{(n)}{\ldots} \times \Gamma(E) \to C^\infty(M)$, defined by means of $C$ and $\langle .,.\rangle$. Denoting by $\reallywidetilde{\mathcal{C}}^n(E)$ the collection of maps $ \reallywidetilde{C}$, this correspondence defines an isomorphism between  ${\mathcal{C}}(E)$ and  $\reallywidetilde{\mathcal{C}}(E)=\oplus_{n\geq 0}\reallywidetilde{\mathcal{C}}^n(E)$ which is indeed an isomorphism of graded Poisson algebras, as it is proved in Section \ref{Multi-Courant}.

Very recently,  Cueca and Mehta \cite{Cueca-Mehta} showed that there is an isomorphism of graded commutative algebras between $(\mathcal{F}_E, \cdot)$ and  $(\reallywidetilde{\mathcal{C}}(E), \wedge)$, where $\cdot$ and $\wedge$ denote the
 associative graded commutative products of the two Poisson algebras $\mathcal{F}_E$ and $\reallywidetilde{\mathcal{C}}(E)$, respectively. They also remarked that the isomorphism is indeed a Poisson isomorphism, but they don't prove this since they don't exhibit the Poisson bracket on $\reallywidetilde{\mathcal{C}}(E)$.

The main goal of this paper is to give a higher version of the Keller-Waldmann Poisson algebra and define higher multi-Courant algebroids. This means that we consider higher (pre-)multi-Courant brackets on $\Gamma(\wedge^{\geq 1}E)$, and not only on $\Gamma(E)$.  Each higher (pre-)multi-Courant bracket $ \mathfrak{C} \in \mathcal{C}^n(\wedge^{\geq 1}E)$
 can be identified as the extension by derivation  of an element $\reallywidetilde{C} \in \reallywidetilde{\mathcal{C}}^n(E)$.
This construction leads to a graded Poisson algebra $\mathcal{C}(\wedge^{\geq1}E)=\oplus_{n\geq 0}\mathcal{C}^n(\wedge^{\geq1}E)$ with a Poisson bracket $[\![\cdot, \cdot]\!]$ that extends $[\cdot,\cdot]_{_{K\!W}}$.  We prove that  $\reallywidetilde{\mathcal{C}}(E)$ and $\mathcal{C}(\wedge^{\geq1}E)$ are isomorphic graded Poisson algebras and therefore $\mathcal{C}(E)$ and $\mathcal{C}(\wedge^{\geq1}E)$ are isomorphic too. In Remark \ref{Remark_4_7} we explain why an element of $\mathcal{C}(\wedge^{\geq1}E)$ can not be directly identified with the extension by derivation of an element of ${\mathcal{C}}(E)$.

In literature we find several Courant bracket extensions, in different directions (see \cite{zambon} and references therein). In \cite{zambon} Zambon defines  higher analogues of Courant algebroids, replacing the vector bundle $TM\oplus T^*M$, originally considered in \cite{Courant}, by $TM \oplus \wedge^pT^*M$, $p\geq 0$.  In  an algebraic setting,  Roytenberg \cite{Roytenberg_2009} extends the usual Courant bracket to a $n$-ary bracket on $\Gamma(E)$,
 and each $n$-ary bracket comes with a collection of symbols that control the defect of their skew-symmetry and also the skew-symmetry of the bracket. In \cite{Peddie}, under the perspective of Loday-infinity algebras and using Voronov's derived bracket construction \cite{voronov2}, Peddie defines $n$-ary Dorfman brackets on $\Gamma(E)$ and $C^\infty(M)$. Having started from  the Keller-Waldmann algebra, whose elements are $n$-ary brackets on $\Gamma(E)$, we were led to an extension of Courant algebroid structures on $E \to M$ in two fold: the binary bracket is replaced by a $n$-ary bracket and the latter is a bracket on sections of $\wedge^{\geq 1}E$. Of course, the symbol goes along the bracket.

The paper is organized in the following way. In Section \ref{Preliminaries} we make a very brief summary of Roytenberg's graded Poisson bracket construction \cite{royContemp} and we recall the  Courant algebroid definition. Section \ref{Multi-Courant} is devoted to Keller-Waldmann Poisson algebra where we clarify and detail many aspects  that are not covered in \cite{Keller-Waldmann}. One of them is  the explicit formula for the Poisson bracket  $[\cdot,\cdot]_{_{K\!W}}$ on $\mathcal{C}(E)$, that is not given in \cite{Keller-Waldmann} because  the bracket
 is defined recursively there. To achieve this, we consider the binary case of a bracket introduced in \cite{Rotkiewcz}, built using the interior product of two elements of  $\mathcal{C}(E)$.  We introduce the concept of multi-Courant algebroid on $(E,\langle \cdot, \cdot \rangle)$ as a $n$-ary element $C \in \mathcal{C}(E)$ that is closed under the  bracket $[\cdot,\cdot]_{_{K\!W}}$.  We point out an alternative definition for the Keller-Waldmann Poisson algebra, already presented in \cite{Keller-Waldmann}, that is needed in the remaining sections of the paper. In this setting,  each $C \in \mathcal{C}(E)$ is in a one-to-one correspondence with $\reallywidetilde{C}$, the latter being obtained from $C$ and  $\langle \cdot, \cdot \rangle$. In Section~\ref{HMC section}, we extend the symmetric bilinear form $\langle \cdot, \cdot \rangle$ to $\Gamma(\wedge^{\bullet}E)$ and prove that it coincides with the restriction of $[\cdot,\cdot]_{_{K\!W}}$ to $\Gamma(\wedge^{\geq 1}E)$. Then, we define
  higher (pre-)multi-Courant structures on $(E,\langle \cdot, \cdot \rangle)$. These are multilinear maps from $\Gamma(\wedge^{\geq 1}E) \times \stackrel{(n)}{\ldots} \times \Gamma(\wedge^{\geq 1}E)$ to $\Gamma(\wedge^{\bullet}E)$ which are derivations in each entry,  together with a symbol that takes values on the space of derivations $Der(C^\infty(M),\Gamma(\wedge^{\bullet}E) )$. All these data should satisfy some compatible conditions involving $\langle \cdot, \cdot \rangle$. The extension by derivation in each entry of every $\reallywidetilde{C}$ is a  higher (pre-)multi-Courant structure on $E$.  We prove that higher (pre-)multi-Courant brackets form a graded Poisson algebra $\mathcal{C}(\wedge^{\geq1}E)$ of degree $-2$ which is isomorphic to $\reallywidetilde{\mathcal{C}}(E)$.

   In Section~\ref{CM isomorphism} we see how the higher Keller-Waldmann Poisson algebra $(\mathcal{C}(\wedge^{\geq1}E), \wedge, [\![\cdot, \cdot]\!])$  is related to Roytenberg's Poisson algebra $(\mathcal{F}_E, \cdot, \{\cdot, \cdot \})$. We start by establishing  a Poisson isomorphism  between the Keller-Waldmann Poisson algebra $(\mathcal{C}(E), \wedge, [\cdot,\cdot]_{_{K\!W}} )$ and  Roytenberg's Poisson algebra,  then  we show that the isomorphism introduced in \cite{Cueca-Mehta} gives rise to a Poisson isomorphism between the higher Keller-Waldmann algebra $(\mathcal{C}(\wedge^{\geq1}E), \wedge, [\![\cdot, \cdot]\!] )$ and Roytenberg's Poisson algebra. As a byproduct of the proof, we get a Poisson isomorphism between $\mathcal{F}_E$ and $\reallywidetilde{\mathcal{C}}(E)$.

 \

 \noindent{\bf Notation.} Let $\tau$ be a permutation of $n$ elements, $n \geq 1$; we denote by $\text{sgn}(\tau)$ the sign of $\tau$.  We denote by $Sh(i, n-i)$ the set of $(i, n-i)$-unshuffles, i.e., permutations $\tau$ that satisfy the inequalities
$\tau(1) < \ldots < \tau(i)$  and $\tau(i+1) < \ldots < \tau(n)$.
For a vector bundle $E \to M$, we denote by $\Gamma(\wedge^n E)$ the space of homogeneous $E$-multivectors of degree $n$ and we set $\displaystyle{\Gamma(\wedge^\bullet E):= \oplus_{n\geq 0}\Gamma(\wedge^n E)}$, with $\Gamma(\wedge^0 E)= C^\infty(M)$, and $\displaystyle{\Gamma(\wedge^{\geq 1} E):= \oplus_{n\geq 1}\Gamma(\wedge^n E)}$. For $n<0$, $\Gamma(\wedge^n E)=\{0\}$.

\

\section{Preliminaries} \label{Preliminaries}

\subsection{Graded Poisson bracket} \label{Graded Poisson bracket}

We briefly recall the construction of a graded Poisson algebra introduced in \cite{royContemp}. Let $E\to M$ be a vector bundle
equipped with a fibrewise non-degenerate symmetric bilinear form $\langle \cdot, \cdot \rangle$ and  denote by $E[m]$ the graded manifold obtained by shifting the fibre degree by $m$.
 Let $p^*(T^*[2]E[1])$ be the graded symplectic manifold
which is the pull-back of $T^*[2]E[1]$ by the map \mbox{$p:E[1] \to E[1]\oplus E^*[1]$} defined by $X \mapsto (X, \frac{1}{2}\langle X,.\rangle)$.
 We denote by $\mathcal{F}_E:= \oplus_{n\geq 0}\mathcal{F}_E^n$ the graded algebra of functions on $p^*(T^*[2]E[1])$, with $\mathcal{F}_E^0=C^\infty(M)$ and  $\mathcal{F}_E^1=\Gamma(E)$ and, consequently, $\Gamma(\wedge^n E)\subset \mathcal{F}_E^n$.
 The graded algebra $\mathcal{F}_E$ is equipped with
the canonical Poisson bracket $\{\cdot,\cdot\}$ of degree $-2$, determined by the graded symplectic structure, so that we have a graded Poisson algebra structure on $\mathcal{F}_E$.
The Poisson bracket of functions of degrees $0$ and $1$ is given by
$$\{f,g\}=0,\; \; \{f, e\}=0 \quad {\hbox{and}} \quad \{e,e'\}=\langle e,e' \rangle,$$
for all $e, e'\in \Gamma(E)$ and $f,g \in C^\infty(M)$.

\

\subsection{Courant structures} \label{subsection2.2} Recall that, given
  a vector bundle $E\to M$
equipped with a fibrewise non-degenerate symmetric bilinear form $\langle \cdot,\cdot \rangle$,
  a \emph{Courant structure} on $(E, \langle\cdot,\cdot\rangle)$ is a pair $(\rho, [\cdot,\cdot])$, where  $\rho:E\to TM$ is a morphism of vector bundles called the \emph{anchor},
 and $[\cdot,\cdot]$  is a $\mathbb{R}$-bilinear bracket on $\Gamma(E)$, called
the \emph{Dorfman bracket}, such that
\begin{equation} \label{pre_Courant1}
\rho(u)\cdot\langle v,w\rangle=\langle[u,v],w\rangle +  \langle v,[u,w]\rangle, \quad \quad \rho(u)\cdot \langle v,w\rangle=\langle u, [v,w]+ [w,v]\rangle,
\end{equation}
and
\begin{equation}\label{Leibniz_Courant}
[u,[v,w]] =[[u,v],w] + [v,[u,w]],
\end{equation}
for all $u,v,w \in \Gamma(E)$. The bracket $[\cdot,\cdot]$ equips the space $\Gamma(E)$ of sections of $E$ with a \emph{Leibniz algebra} structure.
Skipping Equation (\ref{Leibniz_Courant}) yields a \emph{pre-Courant structure} on $(E, \langle\cdot,\cdot\rangle)$.

There is a one-to-one  correspondence between pre-Courant structures $(\rho, [\cdot,\cdot])$ on  $(E, \langle \cdot, \cdot \rangle)$ and functions $\Theta \in \mathcal{F}_E^3$, while for Courant structures the function $\Theta$ is such that $\{\Theta, \Theta \}=0$ \cite{royContemp}. In this case, the hamiltonian vector field $X_{\Theta}=\{ \Theta, \cdot \}$ on the graded manifold $p^*(T^*[2]E[1])$ is a homological vector field, and so $\big(p^*(T^*[2]E[1]), X_{\Theta}\big)$ is a $Q$-manifold.

The anchor and Dorfman bracket associated to a given $\Theta\in \mathcal{F}_E^3$ can be defined, for all $e,e' \in \Gamma(E)$ and $f \in C^\infty(M)$, by the derived bracket expressions:
\begin{equation*}
  \rho(e)\cdot f=\{f,\{e,\Theta\}\} \quad {\hbox{and}} \quad [e,e']=\{e',\{e,\Theta\}\}.
\end{equation*}

\medskip

\section{Multi-Courant structures and Keller-Waldmann Poisson algebra} \label{Multi-Courant}

 In this section we deepen the study of the Keller-Waldmann Poisson algebra. We start by recalling the main definitions and results in \cite{Keller-Waldmann} and we give an explicit definition of the Poisson bracket $[\cdot, \cdot]_{_{K\!W}}$ in ${\mathcal{C}}(E)$, which is defined only recursively in \cite{Keller-Waldmann}. Then, we define a Poisson bracket on $\left(\reallywidetilde{\mathcal{C}}(E), \wedge \right)$ which is nothing but the Poisson bracket announced in \cite{Cueca-Mehta}.

\subsection{Multi-Courant structures} \label{section_4.2}
Let $E\to M$ be a vector bundle
equipped with a fibrewise non-degenerate symmetric bilinear form $\langle \cdot,\cdot \rangle$. The next definition is taken from \cite{Keller-Waldmann}, but within a geometrical perspective. $\mathfrak{X}(M)$ denotes the space of vector fields on a manifold $M$.

\begin{defn}\label{n-Courant}
A  \emph{$n$-ary pre-Courant structure} on $(E, \langle\cdot,\cdot\rangle)$ is a multilinear $n$-bracket on $\Gamma(E)$, $n\geq 0$,
$$C: \Gamma(E) \times \stackrel{(n)}{\ldots} \times \Gamma(E) \to \Gamma(E)$$ for which there exists a map $\sigma_C$, called the \emph{symbol} of $C$,
$$ \sigma_C :  \Gamma(E) \times \stackrel{(n-1)}{\ldots} \times \Gamma(E) \to \mathfrak{X}(M),$$
such that for all $e, e', e_1, \ldots, e_{n-1} \in \Gamma(E)$, we have
\begin{equation} \label{def_pre-Courant_1}
\sigma_C(e_1, \ldots, e _{n-1})\cdot \langle e, e' \rangle= \langle C(e_1, \ldots, e _{n-1},e), e' \rangle + \langle e, C(e_1, \ldots, e _{n-1}, e') \rangle
\end{equation}
and, for $n\geq 2$ and $1\leq i \leq n-1$, the following $n-1$ conditions hold:

\begin{align}\label{def_pre-Courant_2}
\langle C(e_1, \ldots, e_i, e_{i+1}, \ldots, e _{n})+ C(e_1, \ldots, e_{i+1}, e_i, \ldots, e _{n}), e \rangle &  \nonumber\\ =\sigma_C (e_1, \ldots,\widehat{e_{i}}, \widehat{e_{i+1}}, \ldots, e _{n}, e) \cdot \langle e_i, e_{i+1} \rangle,
\end{align}
 with $\widehat{e_i}$ meaning the absence of $e_i$.  A $0$-ary pre-Courant structure is simply an element $e \in \Gamma(E)$ (with vanishing symbol).
The triple $(E, \langle \cdot,\cdot \rangle, C)$ is called an \emph{$n$-ary pre-Courant algebroid}.  When we don't want to specify the arity of $C$, we call it a \emph{pre-multi-Courant} structure and the triple $(E, \langle \cdot,\cdot \rangle, C)$ is a \emph{pre-multi-Courant algebroid}.
\end{defn}

For $n=1$, $C$ is derivative endomorphism \cite{YKS-Mackenzie} with symbol $\sigma_C \in \mathfrak{X}(M)$.
When $n=2$, conditions (\ref{def_pre-Courant_1}) and (\ref{def_pre-Courant_2}) coincide with (\ref{pre_Courant1}) for
$\rho=\sigma_C$. So, as it would be expected, Definition~\ref{n-Courant} generalizes the notion of pre-Courant structure on $(E, \langle\cdot,\cdot\rangle)$.

We denote by $\mathcal{C}^{n+1}(E)$ the space of all $n$-ary pre-Courant structures on $E$ and set
$$\mathcal{C}(E)=\oplus_{n\geq 0}\mathcal{C}^n(E),$$
with $\mathcal{C}^0(E)=C^\infty(M)$ and $\mathcal{C}^1(E)=\Gamma(E)$.

\

\begin{rem}  \label{remark_3.2}
%A $0$-ary pre-Courant structure is simply an element $e \in \Gamma(E)$ (without symbol).
 The symbol $\sigma_C$ of $C\in \mathcal{C}^{n+1}(E)$ is uniquely determined by $C$ ~\cite{Keller-Waldmann}.  The uniqueness of $\sigma_C$ allows to consider an extension of $C$, also denoted by $C$, on the graded space $\Gamma(\wedge^{\leq 1}E)=C^\infty(M) \oplus \Gamma(E)$, where $f\in C^\infty(M)$ has degree $0$ and $e \in \Gamma(E)$ has degree $1$. The extension of $C$ is a degree $1-n$ bracket,
\begin{equation*}
C:  \Gamma(\wedge^{\leq 1}E) \times  \stackrel{(n)}{\ldots} \times   \Gamma(\wedge^{\leq 1}E) \to  \Gamma(\wedge^{\leq 1}E),
\end{equation*}
 with symbol
\begin{equation*}
\sigma_C :  \Gamma(\wedge^{\leq 1}E) \times \stackrel{(n-1)}{\ldots} \times \Gamma(\wedge^{\leq 1}E) \to \mathfrak{X}(M),
\end{equation*}
such that, for all $e_i \in \Gamma(E)$ and $f\in C^\infty(M)$,
\begin{equation} \label{C(f,g)}
 C(e_1, \ldots, e _{n-1},f)= \sigma_C(e_1, \ldots, e _{n-1})\cdot f
\end{equation}
and
$$ C(e_1, \ldots, \stackrel{\stackrel{i}{\downarrow}}{f}, \dots, e_{n-1})=C(e_1, \ldots,  \stackrel{\stackrel{j}{\downarrow}}{f}, \dots, e_{n-1}),$$
for all $1 \leq i,j \leq n$.

By degree reasons, $C$ vanishes when applied to at least two functions,
\begin{equation*}
C(e_1, \ldots, f, \ldots, g, \ldots, e_{n-2})=0.
\end{equation*}
Assuming that $\sigma_C$ vanishes when applied to at least one function,
$$\sigma_C(e_1, \ldots, f, \ldots, e_{n-2})=0, $$
 Equations (\ref{def_pre-Courant_1}) and (\ref{def_pre-Courant_2}), with the obvious adaptations, are satisfied.
\end{rem}

\subsection{Keller-Waldmann Poisson algebra} \label{KW algebra}

Given $C \in \mathcal{C}^{n+1}(E)$, $n\geq 1$, and $e \in \Gamma(E)$,  we denote by $\imath_eC$ the element of $\mathcal{C}^{n}(E)$ defined by
\begin{equation} \label{i_eC}
\imath_eC (e_1, \ldots, e_{n-1})=C(e,e_1, \ldots, e_{n-1}),
\end{equation}
for all $e_1, \ldots, e_{n-1} \in \Gamma(E)$, with symbol  given by
$$\sigma_{\imath_eC}(e_1, \ldots, e_{n-2})=\sigma_{C}(e,e_1, \ldots, e_{n-2}).$$
If $C=e_1$, $\imath_e e_1= \langle e, e_1 \rangle$ and  we set $\imath_e f :=0$.

If we consider the  extension of $C$ as in Remark~\ref{remark_3.2} we may define, for $f \in C^\infty(M)$,
\begin{equation*} \label{i_fC}
\imath_fC(e_1, \ldots, e_{n-1})=C(f, e_1, \ldots, e_{n-1})=\sigma_C(e_1, \ldots, e_{n-1})\cdot f,
\end{equation*}
for all $e_1, \ldots, e_{n-1} \in \Gamma(E)$.

\

The space $\mathcal{C}(E)$ is endowed with an associative graded commutative product $\wedge$ of degree zero \cite{Keller-Waldmann} defined as follows:
$$
\begin{cases}
&f \wedge g =fg=g\wedge f\\
& f \wedge e = f e= e \wedge f,
\end{cases}
$$
for all $f,g \in C^\infty(M)$ and $e \in \Gamma(E)$, and such that, for all $e \in \Gamma(E)$, $\imath_e$ is a derivation of $(\mathcal{C}(E), \wedge)$\footnote{Our signs are different from those in \cite{Keller-Waldmann} and coincide with \cite{A10} and \cite{Cueca-Mehta}.}:
\begin{equation*}  \label{derivation_i_e}
\imath_e(C_1 \wedge C_2)=\imath_e C_1 \wedge C_2 + (-1)^n C_1 \wedge \imath_e C_2,
\end{equation*}
for all $C_1 \in \mathcal{C}^n(E)$ and $C_2 \in \mathcal{C}(E)$. For $C_1 \in \mathcal{C}^n(E)$ and $C_2 \in \mathcal{C}^m(E)$, with $n,m \geq 1$,
$C_1 \wedge C_2$ is equivalently given by \cite{Keller-Waldmann}:
\begin{align}  \label{exterior_product_E}
 C_1 \wedge C_2&\left(e_{1},\ldots,e_{n+m-1}\right) = \nonumber \\ & \sum_{\tau \in Sh(n, m-1)}\text{sgn}(\tau)
  \left\langle C_1\left(e_{\tau(1)},\ldots,e_{\tau(n-1)}\right), e_{\tau(n)} \right\rangle  C_2\left(e_{\tau(n+1)},\ldots,e_{\tau(n+m-1)}\right) \nonumber  \\
  &+(-1)^{nm} \sum_{\tau \in Sh(m,n-1)}\text{sgn}(\tau)\nonumber \\ &\left\langle C_2 \left(e_{\tau(1)},\ldots,e_{\tau(m-1)}\right), e_{\tau(m)} \right\rangle
  C_1\!\left(e_{\tau(m+1)},\ldots,e_{\tau(n+m-1)}\right),
\end{align}
for all $e_{1},\ldots,e_{n+m-1} \in \Gamma(E)$.

  \

The symbol of $C_1 \wedge C_2$ is given by
\begin{align} \label{simbolo_produto}
 \sigma_{C_1 \wedge C_2}&\left(e_{1},\ldots,e_{n+m-2}\right)\cdot f = \nonumber \\ & \sum_{\tau \in Sh(n, m-2)}\text{sgn}(\tau)
  \left\langle C_1\left(e_{\tau(1)},\ldots,e_{\tau(n-1)}\right), e_{\tau(n)} \right\rangle  \sigma_{C_2}\!\left(e_{\tau(n+1)},\ldots,e_{\tau(n+m-2)}\right)\cdot f \nonumber \\
  &+\sum_{\tau \in Sh(n-2, m)}\text{sgn}(\tau)\nonumber \\
  &\Big( \sigma_{C_1}\!\left(e_{\tau(1)},\ldots,e_{\tau(n-2)}\right)\cdot f \Big)
  \left\langle C_2\left(e_{\tau(n-1)},\ldots,e_{\tau(n+m-3)}\right), e_{\tau(n+m-2)} \right\rangle,
\end{align}
for all $e_{1},\ldots,e_{n+m-2} \in \Gamma(E)$ and $f \in C^\infty(M)$.

\

\begin{rem} \label{expression_P}
A homogeneous
  $P \in \Gamma(\wedge^{p} E)$, $P=e_1 \wedge \ldots \wedge e_p$, with $e_i \in \Gamma(E)=\mathcal{C}^{1}(E)$,  can be seen as an element of  $\mathcal{C}^{p}(E)$. From the definition and properties of the interior product, we may obtain an explicit expression for $P(e'_{1}, \ldots , e'_{p-1})$, with $e'_{1}, \ldots , e'_{p-1} \in \Gamma(E)$, by means of products of type $\langle e'_i, e_j\rangle$ (see also Equations (\ref{bilinear_form_explicit}) and (\ref{P})). Furthermore, Equation~(\ref{simbolo_produto}) yields
    $\sigma_P =0$. Conversely, If $C$ is an element of $\mathcal{C}^n(E)$ with $\sigma_C=0$, then $C\in  \Gamma(\wedge^n E)$ (see Lemma~\ref{vanishing_symbol}). For $P, Q \in \Gamma(\wedge^{\bullet} E)$, $P \wedge Q$ is the usual exterior product.
  \end{rem}

\

\begin{defn} \cite{Keller-Waldmann} \label{def_KW_bracket}
The space $\mathcal{C}(E)$ is endowed with a graded Lie bracket of degree $-2$,
$$[\cdot, \cdot]_{_{K\!W}}: \mathcal{C}^n(E)\times \mathcal{C}^m(E)\to \mathcal{C}^{n+m-2}(E),$$
 uniquely defined, for all $f,g \in C^\infty(M)$, $e, e'\in \Gamma(E)$, $D \in \mathcal{C}^2(E)$, $C_1\in \mathcal{C}^n(E)$ and $C_2\in \mathcal{C}(E)$ by \footnote{ Our signs in (iv), (v)  and (\ref{i_e[C1,C2]}) are different from those in \cite{Keller-Waldmann} and coincide with \cite{A10}.},
\begin{enumerate}
\item $[f, g]_{_{K\!W}}=0$,
\item $ [f, e]_{_{K\!W}}= 0= [e, f]_{_{K\!W}}$,
\item $[e, e']_{_{K\!W}}=\langle e, e' \rangle$,
\item $[f,D]_{_{K\!W}}= \sigma_D\cdot f= - [D,f]_{_{K\!W}}$,
\item $[e,C_1]_{_{K\!W}}=(-1)^{n+1}[C_1,e]_{_{K\!W}}=\imath_e C_1$
\end{enumerate}
and, by recursion,
\begin{equation} \label{i_e[C1,C2]}
\imath_e[C_1, C_2]_{_{K\!W}}=[e,[C_1, C_2]_{_{K\!W}}]_{_{K\!W}}=[[e, C_1]_{_{K\!W}}, C_2]_{_{K\!W}} + (-1)^n[C_1, [e,C_2]_{_{K\!W}}]_{_{K\!W}}.
\end{equation}
\end{defn}

\

In \cite{Keller-Waldmann}, it is proved that
\begin{equation} \label{derivation_KW_bracket}
[C_1, C_2 \wedge C_3]_{_{K\!W}}= [C_1,C_2]_{_{K\!W}} \wedge C_3+ (-1)^{nm} C_2 \wedge [C_1,C_3]_{_{K\!W}},
\end{equation}
for all $C_1\in \mathcal{C}^n(E)$, $C_2\in \mathcal{C}^m(E)$ and $C_3 \in \mathcal{C}(E)$. Summing up, we have:

\begin{prop}\cite{Keller-Waldmann}  \label{graded_Poisson_algebra}
The triple $(\mathcal{C}(E), \wedge, [\cdot, \cdot]_{_{K\!W}} )$ is a graded Poisson algebra of degree $-2$, that we call Keller-Waldmann Poisson algebra.
\end{prop}

From Remark~\ref{expression_P}, Definition~\ref{def_KW_bracket} and Equation (\ref{derivation_KW_bracket}), we have:

\begin{cor}  \label{symbol_(P,Q)}
For all $P \in \Gamma(\wedge^p E)$ and $Q \in \Gamma(\wedge^q E)$, $\sigma_{[P,Q]_{_{K\!W}}}=0$.
\end{cor}

 \

Let $V$ be a vector space and set $\mathfrak{g}= V^{\otimes(n-1)}$, for a fixed $n \in \mathbb N$. We denote by $\mathfrak{L}^p$ the space of linear maps from $\mathfrak{g}^{\otimes p} \otimes V$ to $V$ and set $\mathfrak{L}=\oplus_{p\geq0} \mathfrak{L}^p$, with $\mathfrak{L}^0=\mathfrak{g}$. In \cite{Rotkiewcz} a bilinear bracket of degree zero on $\mathfrak{L}$,
$$[\cdot, \cdot]^{n\mathfrak{L}}:\mathfrak{L}^p \times \mathfrak{L}^q \to \mathfrak{L}^{p+q},$$ was introduced. We don't need its explicit definition which can be found in \cite{Rotkiewcz}. However, the important feature of $[\cdot, \cdot]^{n\mathfrak{L}}$ in the present work is that, since $[\cdot, \cdot ]_{_{K\!W}}$ is nothing but $-[\cdot, \cdot]^{2\mathfrak{L}}$, we may have an explicit expression for $[\cdot, \cdot ]_{_{K\!W}}$ that is not given in Definition~\ref{def_KW_bracket}, where the bracket is defined  recursively.

\

Given $C_1 \in  \mathcal{C}^n(E)$ and $C_2 \in  \mathcal{C}^m(E)$, $n,m\geq 1$, the definition in \cite{Rotkiewcz} yields

\begin{equation} \label{explicit_KW}
[C_1, C_2]_{_{K\!W}}= \imath_{C_1}C_2- (-1)^{nm} \imath_{C_2}C_1,
\end{equation}
with $\imath_{C_2}C_1 \in \mathcal{C}^{n+m-2}(E)$ defined, for all $e_1, \ldots, e_{n+m-3} \in \Gamma(E)$, as follows:
\begin{equation}\label{i_C2_C1}
\begin{aligned}
\imath_{C_2}C_1 (e_1, \ldots, e_{n+m-3})= \sum& {\rm{sgn}}(J,I)\,(-1)^t  \\
&C_1(e_{i_{1}}, \ldots, e_{i_{t}}, C_2(e_{j_{1}}, \ldots, e_{j_{m-1}}), e_{i_{t+1}}, \ldots, e_{i_{n-2}}),%\nonumber
\end{aligned}
\end{equation}
where the sum is over all shuffles $I=\{i_1 < \ldots < i_{n-2}\} \subset \{1, \ldots, n+m-3 \}=N$. The $j's$ and $t$ are defined by $\{j_1 < \ldots < j_{m-1} \}=N\backslash I$, $i_{t+1}=j_{m-1}+1$ or, in case $j_{m-1}=n+m-3$, $t:=n-2$. The pair $(J,I)$ denotes the permutation $(j_1, \ldots , j_{m-1}, i_1, \ldots , i_{n-2} )$ of $N$.
When $C_2=e \in  \mathcal{C}^1(E)=\Gamma(E)$,
$\imath_{e}C_1$ is given by Equation~(\ref{i_eC}).

\begin{lem} \label{interior_product}
The interior product $\imath_{C_2}C_1 \in \mathcal{C}^{n+m-2}(E)$ defined in Equation (\ref{i_C2_C1}) is equivalently given by
\begin{equation}\label{novo_i_C2_C1}
\begin{aligned}
&\imath_{C_2}C_1 (e_1, \ldots, e_{n+m-3}) = \ds{\sum_{k= m-1}^{n+m-3} \,\,  \sum_{\tau \in Sh(k-(m-1), m-2)} \, {\rm{sgn}}(\tau) (-1)^{mk}} \\
& C_1(e_{\tau(1)}, \ldots, e_{\tau(k-(m-1))}, C_2(e_{\tau(k-(m-2))}, \ldots, e_{\tau(k-1)}, e_{k} ), e_{k+1}, \ldots, e_{n+m-3}),%\nonumber
\end{aligned}
\end{equation}
 with $C_1 \in  \mathcal{C}^n(E)$ and $C_2 \in  \mathcal{C}^m(E)$, $m\geq 1$.
\end{lem}
\begin{proof}
  We need to prove that (\ref{i_C2_C1}) can be rewritten as  (\ref{novo_i_C2_C1}). Let us consider a permutation $(i_{1}, \ldots, i_{t},j_{1}, \ldots, j_{m-2}, j_{m-1}, i_{t+1}, \ldots, i_{n-2})$ of $N=\{1, \ldots, n+m-3 \}$, as in  (\ref{i_C2_C1}). It is easy to see that the last $n-t-2$ permuted indices, $( i_{t+1}, \ldots, i_{n-2})$, must coincide with the last $n-t-2$ elements of $N$: $(t+m, \ldots, n+m-3)$. Then,
  $$(i_{1}, \ldots, i_{t},j_{1}, \ldots, j_{m-1}, i_{t+1}, \ldots, i_{n-2})=(i_{1}, \ldots, i_{t},j_{1}, \ldots, j_{m-1}, t+m,\ldots, n+m-3).$$
  Moreover, the index $t$ in  (\ref{i_C2_C1}) can be equivalently defined
 % \centerline{``\ $i_{t+1}=j_{s-1}+1$ or, in case $j_{s-1}=r+s-3$, $t:=r-2$\ ''}\\[2mm]
 by setting
  $$j_{m-1}=t+m-1.$$
  Then, permutations $(i_{1}, \ldots, i_{t},j_{1}, \ldots, j_{m-2}, j_{m-1}, i_{t+1}, \ldots, i_{n-2})$ considered in (\ref{i_C2_C1}) can be rewritten as permutations
    \begin{equation}\label{eq_aux_2}
      (i_{1}, \ldots, i_{t},j_{1}, \ldots, j_{m-2}, t+m-1, t+m, \ldots, n+m-3),
    \end{equation}
  where $t$ takes values from $0$ to $n-2$.\footnote{When $t=0$, we have the trivial permutation $(\underbrace{1,\ldots, m-1}_{j_1,\ldots,j_{m-1}},\underbrace{m, \ldots, n+m-3}_{i_1,\ldots,i_{n-2}})$.}
  In addition, if we define the index $k$ by setting $k:=t+m-1$, then  permutation (\ref{eq_aux_2}) corresponds to
  $$(\tau(1), \ldots, \tau(k-(m-1),\tau(k-(m-2)), \ldots, \tau(k-1), k, k+1, \ldots, n+m-3),$$
  where $\tau \in Sh(k-(m-1), m-2)$.

  Finally, we need to rewrite the sign in  (\ref{i_C2_C1}), using the unshuffle $\tau$:
  \begin{align*}
    \text{sgn}(J,I)\times(-1)^t &= \text{sgn}(j_1,\ldots,j_{m+1},i_1,\ldots,i_{n-2})\times(-1)^t \\
    &=\text{sgn}(\tau(k-(m-2)), \ldots, \tau(k-1), k,\tau(1), \ldots, \tau(k-(m-1),\\
    &\quad\quad\quad\quad\quad\quad k+1, \ldots, n+m-3)\times(-1)^{k-(m-1)}\\
    &=\text{sgn}\big((\tau(k-(m-2)), \ldots, \tau(k-1), k,\tau(1), \ldots, \tau(k-(m-1))\big)\\
    &\quad\quad\quad\quad\quad\quad \times(-1)^{k-(m-1)}\\
    &=(-1)^{(k-(m-1))(m-1)}\text{sgn}(\tau(1), \ldots, \tau(k-(m-1),\tau(k-(m-2)),\\
    &\quad\quad\quad\quad\quad\quad  \ldots, \tau(k-1), k)\times(-1)^{k-(m-1)}\\
    &=(-1)^{(k-(m-1))m}\text{sgn}(\tau)=(-1)^{km}\text{sgn}(\tau).
  \end{align*}

  Therefore, we can rewrite  (\ref{i_C2_C1}) as
  \begin{eqnarray*}
    &\imath_{C_2}C_1 (e_1, \ldots, e_{n+m-3}) = \ds{\sum_{k= m-1}^{n+m-3} \,\,  \sum_{\tau \in Sh(k-(m-1), m-2)} \, {\rm sgn}(\tau) (-1)^{km}} \\
    & C_1(e_{\tau(1)}, \ldots, e_{\tau(k-(m-1))}, C_2(e_{\tau(k-(m-2))}, \ldots, e_{\tau(k-1)}, e_{k} ), e_{k+1}, \ldots, e_{n+m-3}).\nonumber
  \end{eqnarray*}

\end{proof}

 Lemma~\ref{interior_product} together with Equation~(\ref{explicit_KW}), provide an explicit  definition of the bracket $[\cdot, \cdot]_{_{K\!W}}$. For example, if  $C_1\in  \mathcal{C}^n(E)$ and $C_2 \in  \mathcal{C}^m(E)$ with  $n=m=3$, we have:
 \begin{align*}
 [C_1, C_2]_{_{K\!W}}(e_1,e_2,e_3)=C_1(C_2(e_1,e_2), e_3)-C_1(e_1, C_2(e_2,e_3))+ C_1(e_2, C_2(e_1,e_3))\\
 +C_2(C_1(e_1,e_2), e_3)-C_2(e_1, C_1(e_2,e_3))+ C_2(e_2, C_1(e_1,e_3)),
 \end{align*}
 for all $e_1,e_2,e_3 \in \Gamma(E)$.

\

For the sake of completeness,  in the next lemma we give the explicit formula for the symbol of $\imath_{C_2}C_1$.
\begin{lem}  \label{symbol_interior_product}
Given $C_1 \in  \mathcal{C}^n(E)$ and $C_2 \in  \mathcal{C}^m(E)$, the symbol of $\imath_{C_2}C_1 \in \mathcal{C}^{n+m-2}(E)$ is given by
\begin{align*}\label{symbol_i_(C2)C1}
  \sigma_{\imath_{C_2}C_1}&\left(e_{1},\ldots,e_{n+m-4}\right)\cdot f =\sum_{k=m-1}^{n+m-4}\sum_{\tau \in Sh(k-(m-1), m-2)}(-1)^{m\left( k-(m-1)\right)}\text{sgn}(\tau)\\
  &\sigma_{C_1}\left(e_{\tau(1)},\ldots,e_{\tau(k-(m-1))},C_2\left(e_{\tau(k-(m-2))},\ldots,e_{\tau(k-1)}, e_k\right), e_{k+1}, \ldots, e_{n+m-4}\right)\cdot f\nonumber\\[4mm]
  &+\sum_{\tau \in Sh(n-2,m-2)}(-1)^{m(n-2)}\text{sgn}(\tau)\nonumber \\
  &\sigma_{C_1}\left(e_{\tau(1)},\ldots,e_{\tau(n-2)}\right)\cdot \left( \sigma_{C_2}\left(e_{\tau(n-1)},\ldots,e_{\tau(n+m-4)}\right)\cdot f\right),\nonumber
\end{align*}
for all $e_1, \ldots, e_{n+m-4} \in \Gamma(E)$ and $f\in C^\infty(M)$.
\end{lem}

\

 \begin{defn} \label{n-ary pre}
 A  pre-multi-Courant structure  $C\in \mathcal{C}^n(E)$, $n \geq 2$, is a \emph{multi-Courant structure} if
$[C, C]_{_{K\!W}}=0.$ In this case, the triple $(E, \langle \cdot,\cdot \rangle, C)$ is called a \emph{multi-Courant algebroid}.
 \end{defn}
 For $n=3$, a multi-Courant structure is simply the usual Courant structure on $(E, \langle \cdot, \cdot \rangle)$.

 \begin{rem} \label{parity}
Since the bracket $[\cdot, \cdot]_{_{K\!W}}$ is graded skew-symmetric, given $C \in \mathcal{C}^{2k}(E)$, $k \geq 1$,  we always  have
$[C,C]_{_{K\!W}}=0.$
So, all $(2k-1)$-ary pre-Courant structures are $(2k-1)$-ary Courant structures.
\end{rem}

\

 Lemma~\ref{interior_product}, Definition~\ref{n-ary pre} and Remark~\ref{parity} yield the next proposition.

\begin{prop} \label{fecho KW_bracket}
 A  pre-multi-Courant structure $C\in \mathcal{C}^{n}(E)$, with $n$ odd, is a  multi-Courant structure if and only if
\begin{align*}
& \ds{\sum_{k= n-1}^{2n-3} \,\,  \sum_{\tau \in Sh(k-(n-1), n-2)} \, {\rm{sgn}}(\tau) (-1)^{nk}} \\
& C(e_{\tau(1)}, \ldots, e_{\tau(k-(n-1))}, C(e_{\tau(k-(n-2))}, \ldots, e_{\tau(k-1)}, e_{k} ), e_{k+1}, \ldots, e_{2n-3})=0,%\nonumber
\end{align*}
for all $e_i \in \Gamma(E)$, $1\leq i \leq 2n-3$.
\end{prop}

\

\begin{rem}  \label{rem_Filippov}
Let $C\in \mathcal{C}^{n+1}(E)$ be a $n$-ary pre-Courant structure on $(E, \langle \cdot, \cdot \rangle)$. If $C$ satisfies
the \emph{Filippov identity} \cite{Fillippov}:
\begin{equation} \label{Filippov}
C(e_1, \ldots, e _{n-1}, C(e'_1, \ldots, e'_{n}))=\sum_{i=1}^{n}C(e'_1, \ldots, e' _{i-1}, C(e_1, \ldots, e _{n-1},e'_i), e'_{i+1}, \ldots, e'_{n}),\end{equation}
for all $e_1, \ldots, e _{n-1}, e'_1, \ldots, e'_{n} \in \Gamma(E)$,
we say that $C$ is a \emph{$n$-Filippov Courant structure} on $E$
 %and the triple $(E, \langle \cdot,\cdot \rangle, C)$ is called a \emph{$n$-Filippov Courant algebroid}
 \footnote{Equation~(\ref{Filippov}) means that $C(e_1, \ldots, e _{n-1},\cdot)$ is a derivation of $C$.}.
Notice that $n$-Filippov Courant structures are called $n$-Courant structures in \cite{Keller-Waldmann}.

 If $C$ is a $n$-Filippov Courant structure on $(E, \langle \cdot, \cdot \rangle)$, $\Gamma(E)$ is equipped with a $n$-Leibniz algebra structure. Thus, a $2$-Filippov Courant structure on $(E, \langle \cdot, \cdot \rangle)$ is the same as a Courant algebroid structure on $(E, \langle \cdot, \cdot \rangle)$. However, comparing Equation~(\ref{Filippov}) with the identity in Proposition~\ref{fecho KW_bracket}, we see that, for $n\geq 3$, $n$-Filippov algebroids and $n$-ary Courant algebroids are different structures.

  An interesting aspect of the bracket $[\cdot, \cdot]^{n\mathfrak{L}}$ introduced in \cite{Rotkiewcz}, is that it characterizes $n$-Leibniz brackets as  those which are closed with respect to it. Indeed, as it is proved in \cite{Rotkiewcz},
Equation (\ref{Filippov}) is equivalent to
$[C,C]^{n\mathfrak{L}}=0.$
\end{rem}

\subsection{An alternative definition} \label{alternative def}

There is an alternative definition of pre-multi-Courant structure on $(E, \langle \cdot, \cdot \rangle)$ that we shall use in the next sections.

Given $C \in  \mathcal{C}^{n}(E)$, $n \geq 1$, we may define a map
 $$ \reallywidetilde{C} :  \Gamma(E) \times \stackrel{(n)}{\ldots} \times \Gamma(E) \to C^\infty(M)$$
 by setting
\begin{equation} \label{def_tilde_C}
\reallywidetilde{C}(e_1, \ldots, e_{n}):= \langle C(e_1, \ldots, e_{n-1}), e_n \rangle
\end{equation}
and, for $C \in  \mathcal{C}^0(E)=C^\infty(M)$,  $\reallywidetilde{C}=C$. Notice that for ${C} \in \mathcal{C}^{1}(E)$, $\reallywidetilde{C} (e)= \langle C,e \rangle$, for all $e \in \Gamma(E)$.

As it is remarked in \cite{Keller-Waldmann}, Definition~\ref{n-Courant} can be reformulated using the maps $\reallywidetilde{C}$. In particular, $\reallywidetilde{C}$ is $C^\infty(M)$-linear in the last entry and Equations (\ref{def_pre-Courant_1}) and (\ref{def_pre-Courant_2}) are equivalent to
\begin{align}\label{def_pre-Courant_3}
\reallywidetilde{C}(e_1, \ldots, e_i, e_{i+1}, \ldots, e _{n})+ \reallywidetilde{C}(e_1, &\ldots, e_{i+1}, e_i, \ldots, e _{n}) \nonumber\\
 &=\sigma_C (e_1, \ldots,\widehat{e_{i}}, \widehat{e_{i+1}}, \ldots, e _{n}) \cdot \langle e_i, e_{i+1}\rangle.
\end{align}

Let $\reallywidetilde{\mathcal{C}}^n(E)$ be the collection of maps $\reallywidetilde{C}$ defined by (\ref{def_tilde_C}), and set $\reallywidetilde{\mathcal{C}}(E)=\oplus_{n\geq 0} \reallywidetilde{\mathcal{C}}^n(E)$. There is a degree zero product on $\reallywidetilde{\mathcal{C}}(E)$, that we also denote by $\wedge$:
\begin{equation} \label{tilde_wedge}
\reallywidetilde{C_1} \wedge \reallywidetilde{C_2}=\reallywidetilde{C_1 \wedge C_2},\end{equation}
for all $\reallywidetilde{C_1} \in \reallywidetilde{\mathcal{C}}^{m}(E)$ and  $\reallywidetilde{C_2} \in \reallywidetilde{\mathcal{C}}^{n}(E)$, $m,n \geq 0$. Explicitly,
\begin{equation}  \label{product_tilde_C}
\reallywidetilde{C_1} \wedge \reallywidetilde{C_2}(e_1, \ldots, e_{m+n})= \sum_{\tau \in Sh(m, n)}\text{sgn}\,(\tau)\,\reallywidetilde{C_1}(e_{\tau(1)},\ldots,e_{\tau(m)})\,\reallywidetilde{C_2}(e_{\tau(m+1)},\ldots,e_{\tau(m+n)}),
\end{equation}
for all $e_{1},\ldots,e_{m+n} \in \Gamma(E)$.
The map
$$\reallywidetilde{\cdot}: \mathcal{C}(E) \to \reallywidetilde{\mathcal{C}}(E), \quad C \in \mathcal{C}^n(E) \mapsto \reallywidetilde{C} \in \reallywidetilde{\mathcal{C}}^n(E),$$
is an isomorphism of graded commutative algebras \cite{Keller-Waldmann}.

We may define a degree $-2$ bracket on $\reallywidetilde{\mathcal{C}}(E)$, by setting
%\footnote{In \cite{Cueca-Mehta} the authors mention the existence of this bracket, but they don't present its explicit definition.}
\begin{equation}  \label{bracket_tilde_C}
\left[\reallywidetilde{C_1}, \reallywidetilde{C_2}\right]_{_{\reallywidetilde{K\!W}}}:=\reallywidetilde{[C_1,C_2]_{_{K\!W}}}
\end{equation}
i.e., given $\reallywidetilde{C_1} \in \reallywidetilde{\mathcal{C}}^{m}(E)$ and  $\reallywidetilde{C_2} \in \reallywidetilde{\mathcal{C}}^{n}(E)$,
\begin{equation*}
\left[\reallywidetilde{C_1}, \reallywidetilde{C_2}\right]_{_{\reallywidetilde{K\!W}}}(e_1, \ldots, e_{m+n-2})= \langle [C_1,C_2]_{_{K\!W}}(e_1, \ldots, e_{m+n-3}),e_{m+n-2}  \rangle
 \end{equation*}
 for all $e_1, \ldots, e_{m+n-2} \in \Gamma(E)$. In Section \ref{CM isomorphism} we shall see that $[\cdot, \cdot]_{_{\reallywidetilde{K\!W}}}$ is the bracket referred in Remark 2.6 of \cite{Cueca-Mehta}.

By construction,
the map $$\reallywidetilde{\cdot}:(\mathcal{C}(E),\wedge, [\cdot, \cdot]_{_{K\!W}} ) \to (\reallywidetilde{\mathcal{C}}(E), \wedge, [\cdot, \cdot]_{_{\reallywidetilde{K\!W}}})$$
is an isomorphism of graded Poisson algebras.

\begin{rem} \label{C_or_tilde_C}
Given $C \in \mathcal{C}^n(E)$, due to (\ref{bracket_tilde_C}) and the non-degeneracy of $\langle \cdot, \cdot \rangle$, we have
$$[C, C]_{_{K\!W}}=0 \; \Leftrightarrow \; \left[\reallywidetilde{C}, \reallywidetilde{C}\right]_{_{\reallywidetilde{K\!W}}}=0$$
and therefore, Definition~\ref{n-ary pre} can be given using either  $\reallywidetilde{C}\in \reallywidetilde{\mathcal{C}}^n(E)$ or  $C \in \mathcal{C}^n(E)$.
\end{rem}

\

\

 \section{Higher multi-Courant structures and higher Keller-Waldmann Poisson algebra} \label{HMC section}
In this section we define higher multi-Courant algebroids, which is the main notion of the paper. Inspired by the generalization of the Lie bracket by the Schouten bracket, we
extend a  pre-multi-Courant structure  $\reallywidetilde{C}\in \reallywidetilde{\mathcal{C}}^n(E)$ on $(E, \langle\cdot,\cdot\rangle)$ to the space $\Gamma(\wedge^{\geq 1} E)$, asking the extension to be a derivation in each entry.

\subsection{Extension of the bilinear form}  \label{subsection_derivatives}
We start by extending the symmetric bilinear form $\langle\cdot,\cdot\rangle$ on $\Gamma(E)$ to
 $\Gamma(\wedge^\bullet E)$ as follows. Given two homogeneous elements $P \in \Gamma(\wedge^p E)$ and $Q \in \Gamma(\wedge^q E)$,  with $p, q \geq 1$, $\langle P,Q \rangle  \in \Gamma(\wedge^{p+q -2} E)$, i.e., $\langle \cdot,\cdot \rangle $ is a degree $-2$ operation. Moreover, $\langle \cdot,\cdot \rangle $ satisfies the following conditions:
\begin{enumerate}
 \item
 $\langle P,Q \rangle = -(-1)^{pq} \langle Q,P \rangle$;
 \item
 $ \langle f,R  \rangle= \langle R,f  \rangle=0$;
 \item
 %\begin{equation*} \label{Jacobi_bilinear_form}
 $\langle P, \langle Q, R \rangle \rangle = \langle \langle P,Q \rangle, R \rangle + (-1)^{pq}\langle Q, \langle P, R \rangle \rangle$
 %\end{equation*}
 \end{enumerate}
and
 \begin{equation} \label{extension_bilinear_form}
 \langle P, Q \wedge R \rangle= \langle P,Q \rangle \wedge R +(-1)^{pq}Q \wedge \langle P,R  \rangle,
 \end{equation}
 for all
  $R \in \Gamma(\wedge^\bullet E)$ and $f \in C^\infty(M)$.
 Extending  by bilinearity,  $\langle\cdot,\cdot\rangle$ is defined in the whole $\Gamma(\wedge^\bullet E)$ and $(\Gamma(\wedge^\bullet E), \langle\cdot,\cdot\rangle )$ is a graded Lie algebra.
Note that $\langle\cdot,\cdot\rangle$ is $C^\infty(M)$-linear in both entries.

\

\begin{lem} \label{lem_extension_bilinear}
Let $P= e_1 \wedge \ldots \wedge e_{p} \in \Gamma (\wedge^{p}E)$ and  $Q= e'_1 \wedge \ldots \wedge e'_{q} \in \Gamma (\wedge^{q}E)$ be two homogeneous elements of $\Gamma(\wedge^{\geq 1} E)$. Then,
\begin{equation} \label{bilinear_form_explicit}
\langle P, Q \rangle= \sum_{s=1}^{q} \sum_{k=1}^{p}(-1)^{k-s+p+1}\langle e_k, e'_s \rangle \widehat{P^k} \wedge \widehat{Q^s},
\end{equation}
where $\widehat{P^k}= e_1 \wedge \ldots \wedge \widehat{e_k} \wedge \ldots \wedge e_{p} \in \Gamma (\wedge^{p-1}E)$ and
$\widehat{Q^s}=  e'_1 \wedge \ldots \wedge \widehat{e'_s} \wedge \ldots \wedge e'_{q} \in \Gamma (\wedge^{q -1}E)$.
\end{lem}

\

For  $P \in \Gamma(\wedge^{p} E)$, $P \in {\mathcal{C}}^p (E)$ and is given by
 \begin{equation}  \label{P}
P(e_1, \ldots, e_{p-1})=\langle e_{p-1}, \ldots, \langle e_2, \langle e_1, P \rangle \rangle \ldots \rangle,
\end{equation}
 while
 $\reallywidetilde{P} \in \reallywidetilde{\mathcal{C}}^p (E)$ is given by
\begin{equation*}  \label{tilde_P}
\reallywidetilde{P}(e_1, \ldots, e_{p})=\langle e_{p}, \ldots, \langle e_2, \langle e_1, P \rangle \rangle \ldots \rangle,
\end{equation*}
for all $e_1, \ldots e_p \in \Gamma(E)$. %So,

\begin{lem} \label{KW_bracket=bilinear_form}
 For $P, Q\in \Gamma(\wedge^{\geq 1}E)$,
 \begin{equation} \label{KW_bracket_multivectors}
 \langle P, Q \rangle=[P, Q]_{_{K\!W}}.
 \end{equation}
 \end{lem}
 \begin{proof}
 Let us prove this result for homogeneous elements $P \in \Gamma(\wedge^p E)$ and $Q \in \Gamma(\wedge^q E)$. We shall use induction on $n=p+q$. For $n=2$, we know from of Definition~\ref{def_KW_bracket}~$(iii)$ that $\langle e, e' \rangle=[e, e']_{_{K\!W}}$, for all $e, e'\in \Gamma(E)$. Now, let us suppose that, for some $k\geq2$ and for all homogeneous elements $P, Q\in \Gamma(\wedge^{\geq 1}E)$, such that $p+q\leq k$, we have $\langle P, Q \rangle=[P, Q]_{_{K\!W}}$. Let us consider $P, Q\in \Gamma(\wedge^{\geq 1}E)$, such that $p+q=k+1$, we need to prove that $\langle P, Q \rangle=[P, Q]_{_{K\!W}}$. Since $p+q=k+1\geq3$, we can suppose, without loss of generality, that $q\geq 2$ and write $Q=\widehat{Q}\wedge e$, for some $e \in \Gamma(E)$. Then, using  (\ref{extension_bilinear_form}) and (\ref{derivation_KW_bracket}), we have
 \begin{align*}
   \langle P, Q \rangle&=\langle P, \widehat{Q}\wedge e \rangle  \\
   &=\langle P,\widehat{Q} \rangle \wedge e +(-1)^{p(q-1)}\widehat{Q} \wedge \langle P,e  \rangle\\
   &=[P, \widehat{Q}]_{_{K\!W}} \wedge e +(-1)^{p(q-1)}\widehat{Q} \wedge [P, e]_{_{K\!W}}\\
   &=[P, \widehat{Q}\wedge e]_{_{K\!W}}\\
   &=[P, Q]_{_{K\!W}}.
 \end{align*}
 \end{proof}

\subsection{Higher Multi-Courant algebroids}
Now we introduce the main notion of this section. By $Der(C^\infty(M),\Gamma(\wedge^{\bullet}E) )$ we denote the space of derivations of $C^\infty(M)$ with values in $\Gamma(\wedge^{\bullet}E)$.

\begin{defn} \label{def_deriv_higher_pre_Courant}
A \emph{higher pre-multi-Courant structure} on $(E, \langle\cdot,\cdot\rangle)$ is a multilinear map
%bracket on $ \Gamma(\wedge^\bullet E)$, $n\geq 0$,
 \begin{equation*}
 \mathfrak{C}: \Gamma(\wedge^{\geq 1}E) \times \stackrel{(n)}{\ldots} \times \Gamma(\wedge^{\geq 1}E) \to \Gamma(\wedge^{\bullet}E), \quad n\geq 0,
 \end{equation*}
of degree $-n$,
 for which  there exists a map $\sigma_\mathfrak{C}$, called the \emph{symbol} of $\mathfrak{C}$,
\begin{equation*}
    \sigma_\mathfrak{C} :  \Gamma(\wedge^{\geq 1}E) \times \stackrel{(n-2)}{\ldots} \times \Gamma(\wedge^{\geq 1}E) \longrightarrow Der(C^\infty(M),\Gamma(\wedge^{\bullet}E) ),
   %(P_1, \ldots, P_{n-2})\longmapsto &\sigma_\mathfrak{C}(P_1, \ldots, P_{n-2}):\  C^{\infty}(M)\to \Gamma(\wedge^{p_1 + \ldots +p_{n-2}-n+2}E)
\end{equation*}
such that
$\mathfrak{C}$ is $C^\infty(M)$-linear in the last entry and the following conditions hold:
\begin{align}
    &\mathfrak{C}(P_1,  \ldots, P_{i}\wedge R,\ldots, P_n)= (-1)^{p_i(p_{i+1}+ \ldots + p_n)} P_i \wedge \mathfrak{C}(P_1, \ldots, R, \ldots,  P_{n}) \label{propertie0_higherC}\\
&\qquad\qquad\qquad\qquad\qquad\qquad+ (-1)^{r(p_{i}+ \ldots + p_n)} R \wedge \mathfrak{C}(P_1, \ldots, P_i, \ldots,  P_{n}),\nonumber
\end{align}
\begin{align} \label{symbol_higherC}
\sigma_\mathfrak{C}(P_1,  \ldots, P_{i}\wedge R,\ldots, P_{n-2})
&= (-1)^{p_i(p_{i+1}+ \ldots + p_{n-2})} P_i \wedge \sigma_\mathfrak{C}(P_1, \ldots, R, \ldots,  P_{n-2})\nonumber\\
&+ (-1)^{r(p_{i}+ \ldots + p_{n-2})} R \wedge \sigma_\mathfrak{C}(P_1, \ldots, P_i, \ldots,  P_{n-2}),
\end{align}
\begin{align}\label{propertie4_higherC}
\mathfrak{C}(P_1, \ldots, P_i,  e, e', P_{i+1} , \ldots,& \, P_{n-2})  +  \mathfrak{C}(P_1, \ldots, P_i,  e', e, P_{i+1} , \ldots, P_{n-2})=\\
&=\sigma_\mathfrak{C}(P_1, \ldots, P_i, P_{i+1} , \ldots, P_{n-2})\left(\langle  e, e' \rangle\right) ,\nonumber
 \end{align}
for all $e, e' \in \Gamma(E)$ and for all homogeneous $P_i \in \Gamma(\wedge^{p_i} E)$, and $R \in \Gamma(\wedge^r E)$, where $p_i\geq 1$, $r\geq 1$ and $1 \leq i \leq n$. For $n=0$, $\mathfrak{C} \in C^\infty(M)$.

The triple $(E, \langle \cdot, \cdot \rangle, \mathfrak{C})$ is called a \emph{higher $n$-ary pre-Courant algebroid} or a \emph{higher pre-multi-Courant algebroid}, if we don't want to specify the arity of $\mathfrak{C}$.
\end{defn}

Notice that, for $P_i \in \Gamma(\wedge^{p_i} E)$, $1 \leq i \leq n-2$, and $f \in C^\infty(M)$,
$$\sigma_\mathfrak{C}(P_1, \ldots, P_{n-2})(f) \in \Gamma(\wedge^{p_1 + \ldots +p_{n-2}-n+2}E).$$

\

\begin{lem}
  If the bilinear form $\langle \cdot,\cdot \rangle$ is full \footnote{The bilinear form is said to be \emph{full} if
$\langle \cdot,\cdot \rangle : \Gamma(E) \times \Gamma(E) \to C^\infty(M)$ is surjective. }, $\sigma_\mathfrak{C}$ is $C^\infty(M)$-linear in the last entry.
\end{lem}
\begin{proof}
  It is a direct consequence of (\ref{propertie4_higherC}) and the fact that $\langle \cdot, \cdot \rangle$ is full and $\mathfrak{C}$ is $C^\infty(M)$-linear in the last entry.
\end{proof}

\

The space of higher $n$-ary pre-Courant structures on $E$ is denoted by $\mathcal{C}^{n}(\wedge^{\geq 1}E)$ and we set
$$\mathcal{C}(\wedge^{\geq 1}E)=\oplus_{n\geq 0}\mathcal{C}^n(\wedge^{\geq 1}E),$$
with  $\mathcal{C}^0(\wedge^{\geq 1}E):=C^\infty(M)$.

\

The alternative definition of pre-multi-Courant structure, introduced in \ref{alternative def}, allows us to construct an example of higher pre-multi-Courant structure.

Given $\reallywidetilde{C} \in \reallywidetilde{\mathcal{C}}^{n}(E)$, we denote by $\overline{\reallywidetilde{C}}$ its extension
by derivation in each entry,
i.e., $\overline{\reallywidetilde{C}}$ and $\reallywidetilde{C}$ coincide on sections of $E$ and, furthermore, $\overline{\reallywidetilde{C}}$ satisfies
 \begin{eqnarray}   \label{bar_C}
 \overline{\reallywidetilde{C}}(P_1,  \ldots, P_{i}\wedge e,\ldots, P_n)
&=& (-1)^{p_i(p_{i+1}+ \ldots + p_n)} P_i \wedge \overline{\reallywidetilde{C}}(P_1, \ldots, e, \ldots,  P_{n}) \nonumber \\ &&+ (-1)^{p_{i}+ \ldots + p_n} e \wedge  \overline{\reallywidetilde{C}}(P_1, \ldots, P_i, \ldots,  P_{n}),
  \end{eqnarray}
for all homogeneous $P_i \in \Gamma(\wedge^{p_i} E), \,p_i \geq 1, \, 1 \leq i \leq n$, and  $e \in \Gamma(E)$.  For $f \in \reallywidetilde{\mathcal{C}}^{0}(E)=C^\infty(M)$, we set $\bar{f}=f$.
 Moreover, we associate to $\overline{\reallywidetilde{C}}$ the map
\begin{equation*}
 \sigma_{\overline{\reallywidetilde{C}}} :  \Gamma(\wedge^{\geq 1}E) \times \stackrel{(n-2)}{\ldots} \times \Gamma(\wedge^{\geq 1}E) \to Der(C^\infty(M),\Gamma(\wedge^{\bullet}E) ),\; \, n\geq2,
 \end{equation*}
 that coincides with $\sigma_C$ on sections of $E$ and
   is the extension by derivation in each entry of $\sigma_C$,  i.e., for all $f \in C^\infty(M)$,
    \begin{eqnarray}   \label{symbol_bar_C}
  \sigma_{\overline{\reallywidetilde{C}}}(P_1, \ldots, P_i \wedge e, \ldots, P_{n-2})\,(f)
&=& (-1)^{p_i(p_{i+1}+ \ldots + p_{n-2})} P_i \wedge \sigma_{\overline{\reallywidetilde{C}}}(P_1, \ldots, e, \ldots,  P_{n-2})\,(f) \nonumber \\ &&+ (-1)^{p_{i}+ \ldots + p_{n-2}} e \wedge  \sigma_{\overline{\reallywidetilde{C}}}(P_1, \ldots, P_i, \ldots,  P_{n-2})\,(f).
  \end{eqnarray}

\begin{lem}
For $\reallywidetilde{C} \in \reallywidetilde{\mathcal{C}}^{n}(E)$, $\overline{\reallywidetilde{C}}$ defined by Equation~(\ref{bar_C}) is an element of $\mathcal{C}^{n}(\wedge^{\geq 1}E)$, with symbol given by Equation~(\ref{symbol_bar_C}).
\end{lem}

\begin{proof}
Applying repeatedly (\ref{bar_C}) (resp. (\ref{symbol_bar_C})), we obtain (\ref{propertie0_higherC}) (resp.  (\ref{symbol_higherC})). Also, it is immediate that $\overline{\reallywidetilde{C}}$ is $C^\infty(M)$-linear in the last entry.

 It remains to prove that, for all $e, e' \in \Gamma(E)$ and for all homogeneous $P_i \in \Gamma(\wedge^{p_i} E)$, $p_i\geq 1$, we have
\begin{align}\label{eq_aux_3}
\overline{\reallywidetilde{C}}(P_1, \ldots, P_i,  e, e', P_{i+1} , \ldots,& P_{n-2})  +  \overline{\reallywidetilde{C}}(P_1, \ldots, P_i,  e', e, P_{i+1} , \ldots, P_{n-2})=\\
&=\sigma_{\overline{\reallywidetilde{C}}}(P_1, \ldots, P_i,  P_{i+1} , \ldots, P_{n-2})\left(\langle  e, e' \rangle\right) .\nonumber
 \end{align}
Let us prove this by induction on $p_1+\ldots+p_{n-2}$. \\
When $p_1+\ldots+p_{n-2}= n-2$, then $p_i=1$, for $i=1,\ldots, n-2$ and (\ref{eq_aux_3}) reduces to (\ref{def_pre-Courant_3}), which is satisfied by $\overline{\reallywidetilde{C}}$ and $\sigma_{\overline{\reallywidetilde{C}}}$. Now, suppose that (\ref{eq_aux_3}) is satisfied for all $P_1, \ldots, P_{n-2}$ such that $n-2\leq p_1+\ldots+p_{n-2}\leq k$, for some $k\geq n-2$, and let us prove it for $P_1, \ldots, P_{n-2}$ such that $p_1+\ldots+p_{n-2}=k+1$. Because $k+1\geq n-1$, there is at least one $j\in\{1,\ldots,n-2\}$ such that $p_j\geq 2$ and then we can write $P_j=\widehat{P_j}\wedge u$, with $u \in \Gamma(E)$. Then,
\begin{align*}
  &\overline{\reallywidetilde{C}}(P_1, \ldots, \widehat{P_j}\wedge u,\ldots, e, e', \ldots, P_{n-2})  +  \overline{\reallywidetilde{C}}(P_1, \ldots, \widehat{P_j}\wedge u,\ldots, e', e, \ldots, P_{n-2})=\\
&\qquad\qquad\qquad=(-1)^{(p_j-1)(p_{j+1}+ \ldots + p_{n-2})}\, \widehat{P_j} \wedge \overline{\reallywidetilde{C}}(P_1, \ldots,u,\ldots, e, e', \ldots, P_{n-2})\\
&\qquad\qquad\qquad\ \  + (-1)^{(p_j-1+p_{j+1} \ldots + p_{n-2})}\, u \wedge \overline{\reallywidetilde{C}}(P_1, \ldots, \widehat{P_j},\ldots, e, e', \ldots, P_{n-2})\\
&\qquad\qquad\qquad\ \  + (-1)^{(p_j-1)(p_{j+1}+ \ldots + p_{n-2})}\, \widehat{P_j} \wedge \overline{\reallywidetilde{C}}(P_1, \ldots,u,\ldots, e', e, \ldots, P_{n-2})\\
&\qquad\qquad\qquad\ \  + (-1)^{(p_j-1+p_{j+1} \ldots + p_{n-2})}\, u \wedge \overline{\reallywidetilde{C}}(P_1, \ldots, \widehat{P_j},\ldots, e', e, \ldots, P_{n-2})\\
&\qquad\qquad\qquad=(-1)^{(p_j-1)(p_{j+1}+ \ldots + p_{n-2})}\, \widehat{P_j} \wedge \bigg(\overline{\reallywidetilde{C}}(P_1, \ldots,u,\ldots, e, e', \ldots, P_{n-2})+\bigg.\\
&\qquad\qquad\qquad\qquad\qquad\qquad\qquad\qquad\ \  +\bigg. \overline{\reallywidetilde{C}}(P_1, \ldots,u,\ldots, e', e, \ldots, P_{n-2})\bigg)\\
&\qquad\qquad\qquad\ \  + (-1)^{(p_j-1+p_{j+1} \ldots + p_{n-2})}\, u \wedge \bigg(\overline{\reallywidetilde{C}}(P_1, \ldots, \widehat{P_j},\ldots, e, e', \ldots, P_{n-2})\bigg.\\
&\qquad\qquad\qquad\qquad\qquad\qquad\qquad\qquad\ \  +\bigg. \overline{\reallywidetilde{C}}(P_1, \ldots, \widehat{P_j},\ldots, e', e, \ldots, P_{n-2})\bigg)\\
&\qquad\qquad\qquad=(-1)^{(p_j-1)(p_{j+1}+ \ldots + p_{n-2})}\, \widehat{P_j} \wedge \sigma_{\overline{\reallywidetilde{C}}}(P_1, \ldots,u,\ldots, \widehat{e}, \widehat{e'}, \ldots, P_{n-2})\left(\langle  e, e' \rangle\right)\\
&\qquad\qquad\qquad\ \  + (-1)^{(p_j-1+p_{j+1} \ldots + p_{n-2})}\, u \wedge \sigma_{\overline{\reallywidetilde{C}}}(P_1, \ldots, \widehat{P_j},\ldots, \widehat{e}, \widehat{e'}, \ldots, P_{n-2})\left(\langle  e, e' \rangle\right) \\
&\qquad\qquad\qquad=\sigma_{\overline{\reallywidetilde{C}}}(P_1, \ldots, \widehat{P_j}\wedge u,\ldots, \widehat{e}, \widehat{e'}, \ldots, P_{n-2})\left(\langle  e, e' \rangle\right).
\end{align*}
\end{proof}

\

Next proposition establishes a relation between $\reallywidetilde{\mathcal{C}}(E)$ and $\mathcal{C}(\wedge^{\geq 1}E)$.

\begin{prop}  \label{1_1}
There is a one-to-one correspondence between
$\reallywidetilde{\mathcal{C}}(E)$ and $\mathcal{C}(\wedge^{\geq 1}E)$
such that,
 for all $n \geq 1$,
\begin{eqnarray*}
\overline{\cdot}: \reallywidetilde{\mathcal{C}}^n(E)& \to & \mathcal{C}^n(\wedge^{\geq 1}E)\\
\reallywidetilde{C} & \mapsto & \overline{\reallywidetilde{C}} ,
\end{eqnarray*}
with $\overline{\reallywidetilde{C}}$ given by Equation~(\ref{bar_C}). For $n=0$, $\overline{\cdot}$ is the identity map.
\end{prop}

\begin{proof}
Given $\mathfrak{C} \in \mathcal{C}^n(\wedge^{\geq 1}E)$, its restriction to $\Gamma(E)$ satisfies (\ref{def_pre-Courant_3})
 so that  $\mathfrak{C}|_{\Gamma(E)} \in \reallywidetilde{\mathcal{C}}^n(E)$. It is obvious that $\overline{\mathfrak{C}|_{\Gamma(E)}}=\mathfrak{C}$.

Now, if $\overline{\reallywidetilde{C}}_1=\overline{\reallywidetilde{C}}_2 \in \mathcal{C}^n(\wedge^{\geq 1}E)$, obviously $\overline{\reallywidetilde{C}}_1|_{\Gamma(E)}=\overline{\reallywidetilde{C}}_2|_{\Gamma(E)}$, which means $\reallywidetilde{C}_1=\reallywidetilde{C}_2$.
\end{proof}

\

Having the one-to-one correspondence given by Proposition~\ref{1_1}, and if there is no ambiguity, in the sequel we shall write very often $\overline{\reallywidetilde{C}}$ instead of $\mathfrak{C}$.

\begin{rem} \label{Remark_4_7}
Let us explain why we consider $\overline{\reallywidetilde{C}}$, the extension of $\reallywidetilde{C} \in \reallywidetilde{\mathcal{C}}(E)$ by derivation in each argument, instead of $\overline{C}$, the extension of $C \in \mathcal{C}(E)$ by derivation in each argument. The reason comes from what should be the extension by derivation of Equation (\ref{def_pre-Courant_2}) in Definition \ref{n-Courant}. The corresponding condition that $\overline{C}$ should satisfy is
\begin{multline*}
  \langle \overline{C}(P_1, \ldots, e, e', \ldots, P _{n-3}) + \overline{C}(P_1, \ldots, e', e, \ldots, P_{n-3}), P_{n-2} \rangle \\
  =\sigma_{\overline{C}} (P_1, \ldots,\widehat{e_{i}}, \widehat{e_{i+1}}, \ldots, P _{n-3}, P_{n-2})( \langle e, e' \rangle),
\end{multline*}
for all $e, e' \in \Gamma(E)$ and for all homogeneous $P_i \in \Gamma(\wedge^{p_i} E)$. But in this expression, the right hand side is derivative with respect to each argument $P_i, i=1,\ldots, n-3$ while the left hand side is not. On the contrary, Equation (\ref{propertie4_higherC}) is fully derivative on both sides.
\end{rem}

\

\subsection{Higher Keller-Waldmann Poisson algebra}

The space $\mathcal{C}(\wedge^{\geq 1}E)$ is endowed with an associative graded commutative product  of degree zero, that we denote by $\wedge$ \footnote{Although we use the same notation, this product is not the one defined in $\mathcal{C}(E)$.},
%\footnote{Here we assume that $f \in C^\infty(M)$ has degree zero.}
defined as follows.  Given $\overline{\reallywidetilde{C}}_1 \in \mathcal{C}^r(\wedge^{\geq 1}E)$ and $\overline{\reallywidetilde{C}}_2 \in \mathcal{C}^s(\wedge^{\geq 1}E)$, %$r,s \geq 1$,
 set
\begin{equation} \label{product_til_bar}
\overline{\reallywidetilde{C}}_1 \wedge \overline{\reallywidetilde{C}}_2:= \overline{\reallywidetilde{C}_1 \wedge \reallywidetilde{C}_2},
\end{equation}
where the product $\wedge$ on the right-hand side is the one defined by Equation~(\ref{product_tilde_C}). Using Equation~(\ref{tilde_wedge}), we may write
$$\overline{\reallywidetilde{C}}_1 \wedge \overline{\reallywidetilde{C}}_2=\overline{\reallywidetilde{C_1 \wedge C_2}}.$$

\

The space $\mathcal{C}(\wedge^{\geq 1}E)$ is endowed with the following  bracket of degree $-2$,
\begin{eqnarray*}
[\![\cdot, \cdot]\!]: \mathcal{C}^r(\wedge^{\geq 1} E) \times \mathcal{C}^s(\wedge^{\geq 1}E)&  \to & \mathcal{C}^{r+s-2}(\wedge^{\geq 1} E)\\
(\overline{\reallywidetilde{C}}_1,\overline{\reallywidetilde{C}}_2)& \mapsto &  \left[\!\!\left[\,\overline{\reallywidetilde{C}}_1,\overline{\reallywidetilde{C}}_2\, \right]\!\!\right]:=\overline{\left[{\reallywidetilde{C}}_1, {\reallywidetilde{C}}_2\right]_{_{\reallywidetilde{K\!W}}}}.
\end{eqnarray*}

As a consequence of Equation (\ref{bracket_tilde_C}) and Lemma \ref{KW_bracket=bilinear_form}, we have:

 \begin{lem}
 For $P, Q\in \Gamma(\wedge^{\geq 1}E)$, $$\left[\!\!\left[\,\overline{\reallywidetilde{P}}, \overline{\reallywidetilde{Q}}\, \right]\!\!\right]=\overline{\reallywidetilde{[P, Q]_{_{K\!W}}}}= \overline{\reallywidetilde{\langle P, Q \rangle}}.$$
 \end{lem}

\

\begin{thm}
The triple $(\mathcal{C}(\wedge^{\geq 1}E), \wedge, [\![\cdot, \cdot]\!] )$ is a graded Poisson algebra of degree $-2$ which is isomorphic to $\left(\reallywidetilde{\mathcal{C}}(E), \wedge, [\cdot, \cdot]_{_{\reallywidetilde{K\!W}}} \right)$.
\end{thm}

\begin{proof}
Bilinearity and graded skew-symmetry of $[\![\cdot, \cdot]\!]$ are obvious. Let us take $\overline{\reallywidetilde{C}}_i \in \mathcal{C}(\wedge^{\geq 1}E)$, $i=1,2,3$. Since
$$\left[\!\!\left[\left[\!\!\left[\,\overline{\reallywidetilde{C}}_1, \overline{\reallywidetilde{C}}_2\, \right]\!\!\right],\overline{\reallywidetilde{C}}_3\, \right]\!\!\right]=\left[\!\!\left[\,\, \overline{[{\reallywidetilde{C}}_1,{\reallywidetilde{C}}_2]_{_{K\!W}}} ,\overline{\reallywidetilde{C}}_3\,\right]\!\!\right]= \overline{\left[\left[{\reallywidetilde{C}}_1, {\reallywidetilde{C}}_2\right]_{_{\reallywidetilde{K\!W}}},{\reallywidetilde{C}}_3\right]_{_{\reallywidetilde{K\!W}}}},$$
the graded Jacobi identity of $[\![\cdot, \cdot]\!]$  follows from the graded Jacobi identity of $[\cdot, \cdot]_{_{\reallywidetilde{K\!W}}}$.
Analogously for the Leibniz rule, since
$$\left[\!\!\left[\,\overline{\reallywidetilde{C}}_1, \overline{\reallywidetilde{C}}_2 \wedge \overline{\reallywidetilde{C}}_3\, \right]\!\!\right]=\left[\!\!\left[\,\overline{\reallywidetilde{C}}_1,  \overline{{\reallywidetilde{C}}_2 \wedge {\reallywidetilde{C}}_3}\, \right]\!\!\right]=\overline{\left[{\reallywidetilde{C}}_1, {\reallywidetilde{C}}_2 \wedge {\reallywidetilde{C}}_3\right]_{_{\reallywidetilde{K\!W}}}}.$$

It is now obvious that
$$\overline{\cdot}:\left(\reallywidetilde{\mathcal{C}}(E), \wedge, [\cdot, \cdot]_{_{\reallywidetilde{K\!W}}} \right) \to \left(\mathcal{C}(\wedge^{\geq 1}E), \wedge, [\![\cdot, \cdot]\!] \right)$$
is an isomorphism of graded Poisson algebras.
\end{proof}

The triple $(\mathcal{C}(\wedge^{\geq 1}E), \wedge, [\![\cdot, \cdot]\!] )$ is called the higher Keller-Waldmann Poisson algebra.

\

Summing up what we have seen so far, the following graded Poisson algebras are isomorphic:
$$\left(\mathcal{C}(E),\wedge, [\cdot, \cdot]_{_{K\!W}} \right) \cong \left(\reallywidetilde{\mathcal{C}}(E), \wedge, [\cdot, \cdot]_{_{\reallywidetilde{K\!W}}} \right) \cong \left(\mathcal{C}(\wedge^{\geq 1}E), \wedge, [\![\cdot, \cdot]\!] \right)$$
\begin{defn} \label{Courant_n}
A higher pre-multi-Courant structure  $\mathfrak{C} \equiv \overline{\reallywidetilde{C}}\in \mathcal{C}^n(\wedge^{\geq 1}E)$, $n \geq 2$, is a higher multi-Courant structure if
$\big[\!\!\big[\,\mathfrak{C}, \mathfrak{C}\, \big]\!\!\big]=0.$ In this case, the triple $(E, \langle \cdot,\cdot \rangle, \mathfrak{C})$ is called a \emph{higher multi-Courant algebroid}.
\end{defn}

 Note that, because the bracket $[\![\cdot, \cdot]\!]$ is skew-symmetric, all $\mathfrak{C} \in \mathcal{C}^{2k}(\wedge^{\geq 1}E)$, $k\geq 1$, are higher multi-Courant structures.

%%%%%%%%%%%%%%%%%%%%%%%%%%%%%%%%%%%%%%%%%%%%%%%%%%%%%%%%%%%%%%%%%%%%%%%%%%%%%%%%%%%%%%%%%%%%%%%%%%%%%%%%%%%%%%%%%%%
%%%%%%%%%%%%%%%%%%%%%%%%%%%%%%%%%%%%%%%%%%%%%%%%%%%%%%%%%%%%%%%%%%%%%%%%%%%%%%%%%%%%%%%%%%%%%%%%%%%%%%%%%%%%%%%%%%%%

\section{On Cueca-Mehta isomorphism} \label{CM isomorphism}

 In this section we consider the graded Poisson algebra
 $(\mathcal{F}_E, \cdot, \{\cdot,\cdot\})$, with $\mathcal{F}_E=C^\infty\big( p^*\big(T^*[2]E[1]\big)\big)$ (see \ref{Graded Poisson bracket}),
 and we prove that it is isomorphic to the Poisson algebra $\mathcal{C}(E)$ of multi-Courant structures and it is also isomorphic to the Poisson algebra $\mathcal{C}(\wedge^{\geq 1}E)$ of higher multi-Courant structures.

 \subsection{Isomorphism between $\mathcal{F}_E$ and $\mathcal{C}(E)$}
The isomorphism
$$
\reallywidetilde{\Upsilon}: (\mathcal{F}_E, \cdot) \to \left(\reallywidetilde{\mathcal{C}}(E), \wedge\right),$$
introduced in \cite{Cueca-Mehta},  maps $\Theta \in \mathcal{F}_E^n$, $n\geq 1$,  into $\reallywidetilde{\Upsilon}(\Theta)\in \reallywidetilde{\mathcal{C}}^{n}(E)$ given  by
%\footnote{In \cite{Cueca-Mehta} the right-hand side of (\ref{def_Upsilon}) is  $\{e_n,  \ldots, \{e_2, \{e_1, \Theta \}\}\ldots \}$.}
\begin{equation} \label{def_Tilde_Upsilon}
\reallywidetilde{\Upsilon}(\Theta)(e_1, e_2, \ldots, e_{n})= \{e_n,  \ldots, \{e_2, \{e_1, \Theta \}\}\ldots \}\in C^\infty(M) ,
\end{equation}
with symbol
$$\sigma_{\reallywidetilde{\Upsilon}(\Theta)}(e_1, \ldots, e_{n-1})\cdot f= \{f, \{e_{n-1},  \ldots,  \{e_1, \Theta \}\}\ldots \},$$
for all $e_1, \ldots, e_{n} \in \Gamma(E)$ and $f \in C^\infty(M)$. For $n=0$, and  for all $f\in \mathcal{F}^0_E=C^\infty(M)$, $\reallywidetilde{\Upsilon}(f)=f$. Moreover, $\reallywidetilde{\Upsilon}$ is an isomorphism of graded commutative algebras \cite{Cueca-Mehta}:
\begin{equation} \label{def_wedge_Tilde_Upsilon}
\reallywidetilde{\Upsilon}(\Theta \cdot \Theta')=\reallywidetilde{\Upsilon}(\Theta) \wedge \reallywidetilde{\Upsilon}(\Theta'), \quad \Theta, \Theta' \in \mathcal{F}_E.
\end{equation}

The isomorphism $\reallywidetilde{\Upsilon}$ induces an isomorphism $$\Upsilon : (\mathcal{F}_E, \cdot) \to (\mathcal{C}(E), \wedge)$$
that maps $\Theta \in \mathcal{F}_E^n$, $n \geq 1$, into $\Upsilon(\Theta)\in \mathcal{C}^{n}(E)$ defined by
\begin{equation} \label{def_Upsilon tilde}
\langle \Upsilon(\Theta)(e_1, \ldots, e_{n-1}), e_n \rangle= \reallywidetilde{\Upsilon}(\Theta)(e_1, \ldots, e_{n}),
\end{equation}
for all $e_1, \ldots, e_{n} \in \Gamma(E)$,  and $\Upsilon(f)=f$, for all $f \in C^\infty(M)$. Due to the non-degeneracy of $\langle \cdot, \cdot \rangle$, $\Upsilon$ is well-defined and, since $\{ \cdot, \cdot \}$ is also non-degenerate, we have
\begin{equation} \label{def_Upsilon}
\Upsilon(\Theta)(e_1, e_2,  \ldots, e_{n-1})= \{e_{n-1},  \ldots, \{e_2, \{e_1, \Theta \}\}\ldots \} \in \Gamma(E).
\end{equation}
In particular, $\Upsilon(e)=e$, for all $e \in \Gamma(E)$.
The symbol of $\Upsilon(\Theta)$ is given by
\begin{equation} \label{symbol_Upsilon}
\sigma_{\Upsilon(\Theta)}(e_1, \ldots, e_{n-2})\cdot f= \{f, \{e_{n-2},  \ldots,  \{e_1, \Theta \}\}\ldots \},
\end{equation}
for all $f \in C^\infty(M)$.

\begin{rem}
Equations (\ref{def_Upsilon}) and (\ref{symbol_Upsilon}) show that, in $\mathcal{F}_E$, the extension of $C\in \mathcal{C}^n(E)$ considered in Remark~\ref{remark_3.2} and, in particular Equation~(\ref{C(f,g)}),  appears in a natural way.
\end{rem}

\begin{rem}
In the case where $E=A\oplus A^*$, the above mentioned isomorphism  $\Upsilon$ was defined in \cite{A10} recursively by the following procedure. Taking into account that the algebra $(\mathcal{C}(E), \wedge)$ is generated by its terms of degrees $0$, $1$ and $2$ \cite{Keller-Waldmann}, the map $\Upsilon$ is defined in $\mathcal{F}_{A\oplus A^*}^0$, $\mathcal{F}_{A\oplus A^*}^1$ and $\mathcal{F}_{A\oplus A^*}^2$ as follows:

$$\left\{
 \begin{array}{ll}
    \Upsilon(f)=f, & f \in C^\infty(M); \\
    \Upsilon(e)=e, & e \in \Gamma(E); \\
    \Upsilon(\Theta)=\{\cdot, \Theta\}, & \Theta \in \mathcal{F}^2.
  \end{array}
\right.$$
Then, $\Upsilon$ is  extended to  $\mathcal{F}_{A\oplus A^*}$ by linearity, asking that Equation~(\ref{morphism_product})
holds for all $\Theta \in \mathcal{F}_{A\oplus A^*}^n$ and $\Theta' \in \mathcal{F}_{A\oplus A^*}^m$. Using the Leibniz rule and the Jacobi identity for $\{\cdot, \cdot\}$,  we conclude that the isomorphisms introduced in \cite{A10} is the same as the one defined by Equation (\ref{def_Upsilon}).
\end{rem}

 Moreover, $\reallywidetilde{\Upsilon}$ being an isomorphism of graded commutative algebras,  $\Upsilon$ inherits the same property, as it is shown in the next lemma.
\begin{lem} \label{preserves_product}
For every $\Theta \in \mathcal{F}_E^n$ and $\Theta' \in \mathcal{F}_E^m$,
\begin{equation}  \label{morphism_product}
\Upsilon(\Theta \cdot \Theta')= \Upsilon(\Theta) \wedge \Upsilon(\Theta').
\end{equation}
\end{lem}
\begin{proof}
Using (\ref{exterior_product_E}), (\ref{def_wedge_Tilde_Upsilon}) and the $C^\infty(M)$-linearity of $\langle \cdot, \cdot \rangle$ we have, for all $e_1, \ldots, e_{m+n}$,
\begin{align*}
&\left\langle \Upsilon(\Theta) \wedge \Upsilon(\Theta')(e_1, \ldots, e_{m+n-1}),e_{m+n} \right\rangle \\
&=\left\langle  \sum_{\tau \in Sh(m, n-1)}\text{sgn}\,(\tau)\reallywidetilde{\Upsilon}(\Theta)(e_{\tau(1)},\ldots,e_{\tau(m)})\,\Upsilon(\Theta')(e_{\tau(m+1)},\ldots,e_{\tau(m+n-1)}), e_{m+n}\right\rangle\\
& + (-1)^{mn} \left\langle  \sum_{\tau \in Sh(n, m-1)}\text{sgn}\,(\tau)\reallywidetilde{\Upsilon}(\Theta')(e_{\tau(1)},\ldots,e_{\tau(n)})\,\Upsilon(\Theta)(e_{\tau(n+1)},\ldots,e_{\tau(n+m-1)}), e_{m+n}\right\rangle  \\
&= \sum_{\tau \in Sh(m, n-1)}\text{sgn}\,(\tau)\reallywidetilde{\Upsilon}(\Theta)(e_{\tau(1)},\ldots,e_{\tau(m)})\reallywidetilde{\Upsilon}(\Theta')(e_{\tau(m+1)},\ldots,e_{\tau(m+n-1)}, e_{m+n})\\
&+ (-1)^{mn}\sum_{\tau \in Sh(n, m-1)}\text{sgn}\,(\tau)\reallywidetilde{\Upsilon}(\Theta')(e_{\tau(1)},\ldots,e_{\tau(n)})\reallywidetilde{\Upsilon}(\Theta)(e_{\tau(n+1)},\ldots,e_{\tau(n+m-1)}, e_{m+n})\\
&=\sum_{\tau \in Sh(m, n)}\text{sgn}\,(\tau)\reallywidetilde{\Upsilon}(\Theta)(e_{\tau(1)},\ldots,e_{\tau(m)})\reallywidetilde{\Upsilon}(\Theta')(e_{\tau(m+1)},\ldots,e_{\tau(m+n-1)}, e_{\tau(m+n)})\\
&=\reallywidetilde{\Upsilon}(\Theta) \wedge \reallywidetilde{\Upsilon}(\Theta')(e_1, \ldots, e_{m+n})= \reallywidetilde{\Upsilon}(\Theta \cdot \Theta')(e_1, \ldots, e_{m+n})\\
&=\left\langle \Upsilon(\Theta \cdot \Theta')(e_1, \ldots, e_{m+n-1}), e_{m+n} \right \rangle.
\end{align*}
The non-degeneracy of $\langle \cdot, \cdot \rangle$ completes the proof.
\end{proof}

\

Since $\Upsilon(f)=f$, for all $f \in C^\infty(M)$, and $\Upsilon(e)=e$, for all $e \in \Gamma(E)$, Lemma~\ref{preserves_product} yields
\begin{equation} \label{Upsilon=id}
\Upsilon_{|\Gamma(\wedge^nE)}=\textrm{id}, \; n\geq 0.
\end{equation}

\

 \begin{lem} \label{lem_auxiliar}
 Let $\Theta$ be a function in $\mathcal{F}_E^n$ and $e \in \Gamma(E)$. Then,
 $$\Upsilon(\{e, \Theta \})= \imath_e(\Upsilon(\Theta))=[e, \Upsilon(\Theta)]_{_{K\!W}}.$$

 \end{lem}

 \begin{proof}
 For all $e_1,\ldots,e_{n-2}\in \Gamma(E)$, we have
 \begin{align*}
   \Upsilon(\{e, \Theta \})(e_1, \ldots, e_{n-2})&=\{e_{n-2}, \ldots,\{e_1,\{e,\Theta\}\}\ldots\}\\
   &=\Upsilon(\Theta)(e, e_1, \ldots, e_{n-2})\\
   &=\imath_e(\Upsilon(\Theta))(e_1, \ldots, e_{n-2}).
 \end{align*}
Thus, $$\Upsilon(\{e, \Theta \})= \imath_e(\Upsilon(\Theta)).$$
 \end{proof}

\begin{prop}  \label{morf_alg_Lie}
Let $\Theta \in \mathcal{F}_E^n$ and $\Theta' \in \mathcal{F}_E^m$, $n,m \geq 0$. Then,
\begin{equation} \label{eq_morf_alg_Lie1}
\Upsilon(\{ \Theta, \Theta' \})= [\Upsilon(\Theta), \Upsilon(\Theta')]_{_{K\!W}}.
\end{equation}
\end{prop}

\begin{proof}
The proof is done by induction on the sum $n+m$ of degrees of $\Theta$ and  $\Theta'$. First, let us prove directly that (\ref{eq_morf_alg_Lie1}) holds for all possible cases such that $n+m\leq 2$.
For $f,g \in C^\infty(M)$, $e \in \Gamma(E)$ and $\delta\in \mathcal{F}_E^2$,  using Definition~\ref{def_KW_bracket} and (\ref{Upsilon=id}),  we have:
\begin{enumerate}
\item
$\Upsilon (\{f,g\})=0=[\Upsilon(f), \Upsilon(g))]_{_{K\!W}};$
\item
$\Upsilon (\{f,e\})=0=[\Upsilon(f), \Upsilon(e)]_{_{K\!W}};$
\item
$\Upsilon (\{e, e'\})=\Upsilon (\langle e, e' \rangle)=\langle e, e' \rangle =[e,e']_{_{K\!W}}=[\Upsilon(e), \Upsilon(e')]_{_{K\!W}};$
\item
$\Upsilon (\{f, \delta\})=\{f, \delta\}=\sigma_{\Upsilon(\delta)}\cdot f= [\Upsilon(f), \Upsilon(\delta)]_{_{K\!W}}$.
\end{enumerate}

Now, let us assume that (\ref{eq_morf_alg_Lie1}) holds for $n+m\leq k$, $k\geq 2$. Take $\Theta \in \mathcal{F}_E^n$ and  $\Theta' \in \mathcal{F}_E^m$, with $n+m=k+1$. For every $e \in \Gamma(E)$, using  (\ref{i_e[C1,C2]}), Lemma~\ref{lem_auxiliar} and the Jacobi identity  of $\{\cdot, \cdot \}$, we have
\begin{eqnarray*}
\imath_e (\Upsilon(\{\Theta, \Theta' \}))&=&\Upsilon(\{e,\{ \Theta, \Theta'\} \})
=\Upsilon\big(\{ \{e,\Theta\}, \Theta' \} \}+ (-1)^n \{\Theta, \{e,\Theta'\}\}\big)\\
&=&[\Upsilon(\{e,\Theta \}), \Upsilon(\Theta')]_{_{K\!W}}+ (-1)^{n} [\Upsilon(\Theta),\Upsilon(\{e,\Theta\})]_{_{K\!W}}\\
&=&[\imath_e(\Upsilon(\Theta)), \Upsilon(\Theta')]_{_{K\!W}}+ (-1)^{n} [\Upsilon(\Theta), \imath_e( \Upsilon(\Theta'))]_{_{K\!W}}\\
&\stackrel{(\ref{i_e[C1,C2]})}{=}&\imath_e[\Upsilon(\Theta), \Upsilon(\Theta')]_{_{K\!W}},
\end{eqnarray*}
where in the third equality we use the induction hypothesis. Since $e \in \Gamma(E)$ is arbitrary, (\ref{eq_morf_alg_Lie1}) is proved.

\end{proof}

We have proved the following.

\begin{thm} \label{th_Poisson_morphism1}
The map
$\Upsilon: (\mathcal{F}_E, \cdot, \{\cdot, \cdot\}) \to (\mathcal{C}(E), \wedge, [\cdot, \cdot]_{_{K\!W}})$ is a degree zero isomorphism of graded Poisson algebras.
\end{thm}

Now, we prove a result announced in Remark~\ref{expression_P}.

\begin{lem} \label{vanishing_symbol}
If $C$ is an element of $\mathcal{C}^n(E)$ with $\sigma_C=0$, then $C\in  \Gamma(\wedge^n E)$.
\end{lem}
\begin{proof}
    Let us consider $C \in \mathcal{C}^n(E)$ such that $\sigma_C=0$. Because $\Upsilon$ is an isomorphism, $C=\Upsilon(\Theta)$, for some $\Theta \in \mathcal{F}^n_E$. Using (\ref{symbol_Upsilon}) and the Jacobi identity for $\{\cdot,\cdot\}$, we have
    \begin{align*}
      0&=\sigma_{C}(e_1, \ldots, e_{n-2})\cdot f= \{f, \{e_{n-2},  \ldots,  \{e_1, \Theta \}\}\ldots \} \\
      &=\{e_{n-2},  \ldots,  \{e_1, \{f,  \Theta \}\}\ldots \},
    \end{align*}
for all $e_1, \ldots, e_{n-2} \in \Gamma(E)$ and $f \in C^\infty(M)$. The non-degeneracy of $\{\cdot,\cdot\}$ implies that  equation above is equivalent to $$\{f, \Theta\}=0, \text{ for all } f \in C^\infty(M),$$
which means (see \cite{royContemp}) that $\Theta \in \Gamma(\wedge^n E)\subset \mathcal{F}^n_E$. Therefore, since $\Upsilon_{|\Gamma(\wedge^n E)}$ is the identity map, we have $C\in  \Gamma(\wedge^n E)$.
\end{proof}
\

\subsection{Isomorphism between $\mathcal{F}_E$ and $\mathcal{C}(\wedge^{\geq 1} E)$}
 For establishing the isomorphism between the graded Poisson algebras $(\mathcal{F}_E, \cdot, \{\cdot,\cdot \})$ and $\left(\mathcal{C}(\wedge^{\geq 1} E),\wedge, [\![\cdot, \cdot]\!] \right)$, we have to consider $\reallywidetilde{\Upsilon}$.

The isomorphism $
\reallywidetilde{\Upsilon}: \mathcal{F}_E \to \reallywidetilde{\mathcal{C}}(E)$, defined by Equation (\ref{def_Tilde_Upsilon}), naturally gives rise to a map
$$
\overline{\reallywidetilde{\Upsilon}}: \mathcal{F}_E  \rightarrow \mathcal{C}(\wedge^{\geq 1} E)$$
such that
$$ \Theta \in \mathcal{F}^n_E  \mapsto \overline{\reallywidetilde{\Upsilon}}(\Theta):= \overline{\reallywidetilde{\Upsilon}(\Theta)}\in \mathcal{C}^n(\wedge^{\geq 1} E),
$$
where $\overline{\reallywidetilde{\Upsilon}(\Theta)}$ is the extension by derivation in each entry of $\reallywidetilde{\Upsilon}(\Theta) \in \reallywidetilde{\mathcal{C}}^n(E)$.

\begin{thm}\label{Upsilon_Isom_Poisson}
The map $\overline{\reallywidetilde{\Upsilon}}: (\mathcal{F}_E, \cdot, \{\cdot,\cdot \})  \rightarrow \left(\mathcal{C}(\wedge^{\geq 1} E), \wedge, [\![\cdot, \cdot]\!] \right)$ is an  isomorphism of graded Poisson algebras.
\end{thm}

\begin{proof}
For every $\Theta \in \mathcal{F}^n_E$ and $\Theta' \in \mathcal{F}^m_E$, we have
\begin{multline*}
\overline{\reallywidetilde{\Upsilon}}(\{ \Theta, \Theta' \})=\overline{\reallywidetilde{\Upsilon}(\{ \Theta, \Theta' \})}= \overline{\reallywidetilde{\Upsilon(\{ \Theta, \Theta' \})}}= \overline{\reallywidetilde{[\Upsilon(\Theta), \Upsilon(\Theta')]_{_{K\!W}}}}\\=\overline{\left[\,\reallywidetilde{\Upsilon}(\Theta), \reallywidetilde{\Upsilon}(\Theta')\, \right]_{_{\reallywidetilde{K\!W}}}}=\left[\!\!\left[\, \overline{\reallywidetilde{\Upsilon}}(\Theta), \overline{\reallywidetilde{\Upsilon}}(\Theta')\, \right]\!\!\right].
\end{multline*}

Moreover,  (\ref{product_til_bar}), (\ref{tilde_wedge}) and (\ref{morphism_product}) yield
$$\overline{\reallywidetilde{\Upsilon}}(\Theta \cdot \Theta')=\overline{\reallywidetilde{\Upsilon}}(\Theta)\wedge \overline{\reallywidetilde{\Upsilon}}(\Theta').$$
\end{proof}

\begin{rem}
We should stress that, although
$$\Upsilon(\Theta)(e_1, \ldots, e_{n-1})= \{e_{n-1},  \ldots, \{e_2, \{e_1, \Theta \}\}\ldots \},$$
for all $e_1, \ldots, e_{n-1} \in \Gamma(E)$, in general
$$\overline{\reallywidetilde{\Upsilon}}(\Theta)(P_1, \ldots, P_{n})\neq\{P_{n},  \ldots, \{P_2, \{P_1, \Theta \}\}\ldots \},$$
for all $P_1, \ldots, P_{n} \in \Gamma(\wedge^{\geq 1} E)$. The difference comes from the fact that the derivation property of the higher multi-Courant structure (see Equation (\ref{bar_C})) do not match the Leibniz property of the (iterated) bracket $\{\cdot,\cdot\}$.

\end{rem}

Notice that from the proof of Theorem \ref{Upsilon_Isom_Poisson}, we get the following result.

\begin{cor}
The map
$\reallywidetilde{\Upsilon} : (\mathcal{F}_E, \cdot, \{\cdot, \cdot\}) \to \left(\reallywidetilde{\mathcal{C}}(E), \wedge, [\cdot, \cdot]_{_{\reallywidetilde{K\!W}}}\right)$ is an  isomorphism of graded Poisson algebras.
\end{cor}

\begin{rem}
The $[\cdot, \cdot ]_{_{\reallywidetilde{K\!W}}}$ bracket, given by Equation~(\ref{bracket_tilde_C}), is the bracket announced, but not explicitly defined, in Remark 2.6 of \cite{Cueca-Mehta}.
\end{rem}
%%%%%%%%%%%%%%%%%%%%%%%%%%%%%%%%%%%%%%%%%%%%%%%%%%%%%%%%%%%%%%%%%%%%%%%%%%%%%%%%%%%%%%%%%%%%%%%%
%%%%%%%%%%%%%%%%%%%%%%%%%%%%%%%%%%%%%%%%%%%%%%%%%%%%%%%%%%%%%%%%%%%%%%%%%%%%%%%%%%%%%%%%%%%%%%%%%%%%%%%%%%%%%%

\

\noindent {\bf Acknowledgments.}  This work was partially supported by the Centre for Mathematics of the University of Coimbra - UIDB/00324/2020, funded by the Portuguese Government through FCT/MCTES.

\

\end{document}